\documentclass[a4paper,11pt]{article}

\usepackage{amsfonts,amsmath,amssymb,amsthm,epsfig,psfrag,verbatim}
\usepackage[all]{xy}
\usepackage[margin=1in,nomarginpar,nohead,nofoot,a4paper,footskip=1cm]{geometry}

\newtheorem{theorem}{Theorem}
\newtheorem*{theorem*}{Theorem}
\newtheorem{definition}{Definition}
\newtheorem{lemma}{Lemma}
\newtheorem*{lemma*}{Lemma}
\newtheorem{proposition}{Proposition}
\newtheorem*{proposition*}{Proposition}
\newtheorem{corollary}{Corollary}
\newtheorem{remark}{Remark}

\def\bx{\mathbf{x}}
\def\by{\mathbf{y}}
\def\bb{\mathbf{b}}
\def\bc{\mathbf{c}}
\def\bX{\mathbf{X}}
\def\bm{\mathbf{m}}
\def\bn{\mathbf{n}}
\def\PP{{\mathbb P}}
\def\EE{{\mathbb E}}
\def\N{{\mathbb N}}
\def\R{{\mathbb R}}
\def\ind{\mathbf{1}}
\def\d{\partial}
\def\ore{{\overrightarrow{e}}}
\def\ole{{\overleftarrow{e}}}
\def\orE{{\overrightarrow{E}}}
\def\ordv{{\overrightarrow{\partial v}}}
\def\oldv{{\overleftarrow{\partial v}}}
\def\orde{{\partial{\overrightarrow e}}}

\def\Pcal{{\mathcal P}}
\def\Rcal{{\mathcal R}}
\def\Dcal{{\mathcal D}}
\def\Qcal{{\mathcal Q}}
\def\Scal{{\mathcal S}}
\def\Gcal{{\mathcal G}}
\def\Ucal{{\mathcal U}}
\def\Fcal{{\mathcal F}}
\def\Hcal{{\mathcal H}}

\def\Mcal{{\mathcal M}}

\def\BGWcal{{\mathcal B\mathcal G\mathcal W}}
\def\bScal{{\mathbf{\mathcal S}}}

\def\lru{{\text{lr}\uparrow}}
\def\tPcal{\widetilde{\mathcal P}}
\def\Pcallc{{\mathcal P}_{\text{lc}}}
\def\tPcallc{\widetilde{\mathcal P}_{\text{lc}}}
\def\balpha{{\boldsymbol{\alpha}}}
\def\bbeta{{\boldsymbol{\beta}}}

\begin{document}
\title{Convergence of Multivariate Belief Propagation, with Applications to Cuckoo Hashing and Load Balancing}

\author{M. Leconte\\Technicolor - INRIA\\mathieu.leconte@inria.fr\and M. Lelarge\\INRIA - \'Ecole Normale Sup\'erieure\\marc.lelarge@ens.fr\and L. Massouli\'e\\Technicolor\\laurent.massoulie@technicolor.com}

\date{}

\maketitle
\begin{abstract}
This paper is motivated by two applications, namely i) generalizations of cuckoo hashing, a computationally simple approach to assigning keys to objects, and ii) load balancing in content distribution networks, where one is interested in determining the impact of content replication on performance.
These two problems admit a common abstraction: in both scenarios, performance is characterized by the maximum weight of a generalization of a matching in a bipartite graph, featuring node and edge capacities.

Our main result is a law of large numbers characterizing the asymptotic maximum weight matching in the limit of large bipartite random graphs, when the graphs admit a {\em local weak limit} that is a tree. This result specializes to the two application scenarios, yielding new results in both contexts.
In contrast with previous results, the key novelty is the ability to handle edge capacities with arbitrary integer values.

An analysis of belief propagation algorithms (BP) with multivariate belief vectors underlies the proof. In particular, we show convergence of the corresponding BP by exploiting monotonicity of the belief vectors with respect to the so-called {\em upshifted likelihood ratio} stochastic order. This auxiliary result can be of independent interest, providing a new set of structural conditions which ensure convergence of BP.

\end{abstract}

\newpage

\section{Introduction}

Belief Propagation (BP) is a popular message-passing algorithm for
determining approximate marginal distributions in Bayesian networks
\cite{pearl88} and statistical physics \cite{Mezard09} or for decoding
LDPC codes \cite{ricurb08}. The popularity of BP stems from its
successful application to very diverse contexts where it has been
observed to converge quickly to meaningful limits \cite{yfw05},
\cite{mmw05}. 
In contrast, relatively few theoretical results are available to prove
rigorously its convergence and uniqueness of its fixed points when the
underlying graph is not a tree \cite{bbcz07}.

In conjunction with the local weak convergence \cite{aldste}, BP has
also been used as an analytical tool to study combinatorial
optimization problems on random graphs: through a study of its fixed
points, one can determine so-called Recursive Distributional Equations
(RDE) associated with specific combinatorial problems. In turn, these
RDEs determine the asymptotic behaviour of solutions to the associated
combinatorial problems in the limit of large instances. Representative
results in this vein concern matchings \cite{bls11}, spanning
subgraphs with degree constraints \cite{Salez11} and orientability of random hypergraphs \cite{Lelarge12}. 

All these problems can be encoded with binary values on the edges of
the underlying graph and these contexts involve BP with scalar
messages. A key step in these results consists in showing monotonicity
of the BP message-passing routine with respect to the input
messages. As an auxiliary result, the analyses of \cite{Salez11} and
\cite{Lelarge12} provide structural monotonicity properties under
which BP is guaranteed to converge (when messages are scalar).

The present work is in line with \cite{Salez11}, \cite{Lelarge12}
and contributes to a rigorous formalization of the cavity method,
originating from statistical physics \cite{mpv87}, \cite{kmrsz}, and
applied here to a generalized matching problem \cite{lp09}. 
The initial motivation is the analysis of generalized matching problems in bipartite graphs with both edge and node capacities. This generic problem has several applications. In particular, it accurately models the service capacity of distributed content delivery networks under various content encoding scenarios, by letting nodes of the bipartite graph represent either contents or servers. It also models problem instances of cuckoo hashing, where in that context nodes represent either objects or keys to be matched. 

Previous studies of these two problems \cite{Lelarge12,Leconte12} essentially required unit edge capacities, which in turn ensured that the underlying BP involved only scalar messages. It is however necessary to go beyond such unit edge capacities to accurately model general server capacities and various content coding schemes in the distributed content delivery network case. The extension to general edge capacities is also interesting in the context of cuckoo hashing when keys can represent sets of addresses to be matched to objects (see Section~\ref{section: cuckoo hashing}).

Our main contribution is Theorem~\ref{th: maximum allocation for bipartite Galton-Watson limits}, a law of large numbers characterizing the asymptotic size of maximum size generalized matchings in random bipartite graphs in terms of RDEs. It is stated in Section~\ref{section: main result}. It is then applied to cuckoo hashing and distributed content delivery networks in Section~\ref{section: applications}, providing generalizations of the results in \cite{Lelarge12} and \cite{Leconte12} respectively. 

Besides obtaining these new laws of large numbers, our results also have algorithmic implications. Indeed to prove Theorem~\ref{th: maximum allocation for bipartite Galton-Watson limits}, in Section~\ref{section: overview} we state Proposition~\ref{prop: monotonicity of local operators}, giving simple continuity and monotonicity conditions on the message-passing routine of BP which guarantee its convergence to a unique fixed-point. This result is shown to apply in the present context for the so-called upshifted likelihood ratio stochastic order. Beyond its application to the present matching problem, this structural result might hold under other contexts, and with stochastic orders possibly distinct from the upshifted likelihood ratio order, to establish convergence of BP in the case of multivariate messages.

\section{Main result}\label{section: main result}
Let $G=(V,E)$ be a finite graph, with additionally an integer vertex-constraint $b_v$ attached to each vertex $v\in V$ and an integer edge-constraint $c_e$ attached to each edge $e\in E$. 

A vector $\bx=(x_e)_{e\in E}\in\N^E$ is called an \emph{allocation} of $G$ if
\begin{align*}
\forall e\in E,\:0\leq x_e\leq c_e\text{ and }\forall v\in V,\:\sum_{e\in\d V}x_e\leq b_v,
\end{align*}
where $\d v$ is the set of edges adjacent to $v$ in $G$. We also write $u\sim v$ when $uv\in E$.

For an allocation $\bx$ of $G$, we define the size $|\bx|$ of $\bx$ as  $|\bx|:=\sum_{e\in E}x_e$, and we denote by $M(G)$ the maximum size of an allocation of $G$. Our aim is to characterize the behaviour of $M(G)/|V|$ for large graphs $G$ in the form of a law of large numbers as $|V|$ goes to infinity.

We focus mainly on sequences of graphs $(G_n)_{n\in\N}$ which converge locally weakly towards Galton-Watson trees $G$. In short (we will explain more in detail later), what this convergence means is that, if we let $R_n$ be a vertex chosen uniformly at random in $G_n$, what $R_n$ sees within any finite graph distance $k$ looks more and more like the $k$-hop neighborhood of the root of a Galton-Watson tree as $n\to\infty$. Such a tree is characterized by a joint law $\Phi\sim(D,W,\{C_i\}_{i=1}^D)$ for respectively the degree, vertex-constraint and adjacent edge-constraints (counted with multiplicity) of the vertices of $G$. We always assume that the graphs are locally finite, i.e. $D<\infty$ a.s. 

To sample a Galton-Watson tree $G$, we first draw a sample from $\Phi$ for the root.
Then we construct at each dangling edge the missing vertex and its other adjacent edges (therefore maybe creating new dangling edges), until no dangling edge remains. Independently for each dangling edge and conditionally on its capacity $c_0$, we draw a sample $(\widetilde D,\widetilde W,\{\widetilde C_i\}_{i=1}^{\widetilde D}|c_0)\sim\widetilde\Phi(\cdot|c_0)$ for the number of other adjacent edges (not counting the dangling edge), the capacity of the vertex, and the other adjacent edge-constraints. Specifically, the distribution $\widetilde\Phi$ is given by $$\widetilde\Phi(\tilde d-1,\tilde b,\{\tilde c_1,\ldots,\tilde c_{\tilde d-1}\}|c_0)=\frac{\Phi(\tilde d,\tilde b,\{c_0,\tilde c_1,\ldots,\tilde c_{\tilde d-1}\})(1+\sum_{i=1}^{\tilde{d}-1}\ind(\tilde c_i=c_0))}{\sum_{(d,b,\{c_1,\ldots,c_{d-1}\})}\Phi(d,b,\{c_0,\ldots,c_{d-1}\})(1+\sum_{i=1}^{d-1}\ind(c_i=c_0))}.$$
The construction above can be extended to bipartite graphs $G=(A\cup B,E)$. In that case, there are two laws $\Phi^A$ and $\Phi^B$ for the characteristics $(D^A,W^A,\{C_i^A\}_{i=1}^{D^A})$  and $(D^B,W^B,\{C_i^B\}_{i=1}^{D^B})$ of vertices in $A$ and $B$ respectively. These verify the consistency relation for all edge capacities $c$:
$$
\frac{1}{\EE[D^A]}\EE\sum_{i=1}^{D^A}\ind(C_i^A=c)=\frac{1}{\EE[D^B]}\EE\sum_{i=1}^{D^B}\ind(C_i^B=c).
$$
The construction then alternates between $\widetilde\Phi^A$ and $\widetilde\Phi^B$ for vertices at even and odd distances from the root.

We define $[z]_x^y=\max\left\{x,\min\{y,z\}\right\}$. Our main result allows to compute the limit $\Mcal(\Phi^A,\Phi^B)$ of $M(G_n)/|A_n|$ when $(G_n)_{n\in\N}$ converges locally weakly towards a bipartite Galton-Watson tree $G=(A\cup B,E)$ defined by $\Phi^A$ and $\Phi^B$:
\begin{theorem}[Maximum allocation for bipartite Galton-Watson limits]\label{th: maximum allocation for bipartite Galton-Watson limits}
Provided $\EE[W^A]$ and $\EE[W^B]$ are finite, the limit $\Mcal(\Phi^A,\Phi^B):=\lim_{n\to\infty}M(G_n)/|A_n|$ exists and equals
$$
\begin{array}{ll}
\Mcal(\Phi^A,\Phi^B)=&\inf
\left\{\EE\left[\min\Bigg\{W^A,\sum_{i=1}^{D^A}X_i(C_i^A)\right\}\right]\\
&+\frac{\EE[D^A]}{\EE[D^B]}\EE\left[\left(W^B-\sum_{i=1}^{D^B}\left[W^B-\sum_{j\neq i}Y_j(C_j^B)\right]_0^{C_i^B}\right)^+\ind\left({W^B<\sum_{i=1}^{D^B}C_i^B}\right)\right]\Bigg\}
\end{array}
$$
where for all $i$, $\left(X_i(c),Y_i(c)\right)_{c\in\N}$ is an independent copy of $\left(X(c),Y(c)\right)_{c\in\N}$, and the infimum is taken over distributions for $\left(X(c),Y(c)\right)_{c\in\N}$ satisfying the RDE
\begin{align*}
Y(c)=\left\{\left[\widetilde W^A-\sum_{i=1}^{\widetilde D^A}X_i(\widetilde C_i^A)\right]_0^c\Bigg|C_0^A=c\right\}; X(c)=\left\{\left[\widetilde W^B-\sum_{i=1}^{\widetilde D^B}Y_i(\widetilde C_i^B)\right]_0^c\Bigg|C_0^B=c\right\}.
\end{align*}
\end{theorem}

\begin{remark}
A similar result holds when the graphs are not bipartite; the limiting tree is then simply a Galton-Watson tree described by a joint distribution $\Phi$. We set $\Phi^A=\Phi^B=\Phi$, and the formula in Theorem~\ref{th: maximum allocation for bipartite Galton-Watson limits} then computes $\lim_{n\to\infty}\frac{2M(G_n)}{|V_n|}=\Mcal(\Phi,\Phi)$.
\end{remark}

\section{Applications}\label{section: applications}

We now apply Theorem~\ref{th: maximum allocation for bipartite Galton-Watson limits} to performance analysis of generalized cuckoo hashing and distributed content-delivery networks.

\subsection{Cuckoo hashing and hypergraph orientability}\label{section: cuckoo hashing}
Cuckoo hashing is a simple approach for assigning keys (hashes) to items. Given an initial collection of $n$ keys, each item is proposed upon arrival two keys chosen at random and must select one of them. Depending on the number $m$ of items and the random choices offered to each item, it may or may not be possible to find such an assignement of items to keys. In the basic scenario, it turns out that such an assignement will be possible with probability tending to $1$ as $m,n\to\infty$ for all $m=\lfloor\tau n\rfloor$ with $\tau<\frac{1}{2}$.

The basic problem can be extended in the following meaningful ways:
\begin{itemize}
\item each item can choose among $h\geq2$ random keys \cite{Dietzfelbinger10,Fountoulakis11,Frieze09};
\item each key can hold a maximum of $k$ items  \cite{Dietzfelbinger05,Cain07,Fernholz07};
\item each item must be replicated at least $l$ times \cite{Gao10,Lelarge12};
\item each (item,key) pair can be used a maximum of $r$ times (not covered previously)
\end{itemize}
the basic setup corresponding to $(h,k,l,r)=(2,1,1,1)$.
We let $\tau^*_{h,k,l,r}$ be the associated threshold, i.e. if $m=\lfloor\tau n\rfloor$ with $\tau<\tau^*_{h,k,l,r}$ then an assignement of items to keys satisfying the conditions above will exist with probability tending to $1$ as $m,n\to\infty$; on the contrary, if $\tau>\tau^*_{h,k,l,r}$, then the probability that such an assignement exists will tend to $0$ as $m,n\to\infty$.

An alternative description of the present setup consists in the following hypergraph orientation problem. For $h\in\N^*$, a $h$-uniform hypergraph is a hypergraph whose hyperedges all have size $h$. We assign marks in $\{0,\ldots,r\}$ to each of the endpoints of a hyperedge. For $l<h$ in $\N^*$, a hyperedge is said to be $(l,r)$-oriented if the sum of the marks at its endpoints is equal to $l$. The in-degree of a vertex of the hypergraph is the sum of the marks assigned to it in all its adjacent hyperedges. For a positive integer $k$, a $(k,l,r)$-orientation of a $h$-uniform hypergraph is an assignement of marks to all endpoints of all hyperedges such that every hyperedge is $(l,r)$-oriented and every vertex has in-degree at most $k$; if such a $(k,l,r)$-orientation exists, we say that the hypergraph is $(k,l,r)$-orientable. We now consider the probability space $\Hcal_{n,m,h}$ of the set of all $h$-uniform hypergraphs with $n$ vertices and $m$ hyperedges, and we denote by $H_{n,m,h}$ a random sample from $\Hcal_{n,m,h}$. In this context, we can interpret Theorem~\ref{th: maximum allocation for bipartite Galton-Watson limits} as follows:

\begin{theorem}[Threshold for $(k,l,r)$-orientability of $h$-uniform hypergraphs]\label{th: application to cuckoo hashing}
Let $h,k,l,r$ be positive integers such that $k,l\geq r$, $(h-1)r\geq l$ and $k+(h-2)r-l>0$ (i.e. at least one of the inequalities among $k\geq r$ and $(h-1)r\geq l$ is strict). We define $\Phi^A$ and $\Phi^B_\tau$ by $(h,l,\{r\})\sim\Phi^A$ and $(\operatorname{Poi}(\tau h),k,\{r\})\sim\Phi^B_\tau$, and
\begin{align*}
\tau^*_{h,k,l,r}=\sup\left\{\tau:\Mcal(\Phi^A,\Phi^B_\tau)<l\right\}.
\end{align*}

Then,
\begin{align*}
\lim_{n\to\infty}\PP\left(H_{n,\lfloor\tau n\rfloor,h}\text{ is }(k,l,r)\text{-orientable}\right)=\left\{\begin{array}{ll}
 1\text{ if }\tau<\tau^*_{h,k,l,r}\\
 0\text{ if }\tau>\tau^*_{h,k,l,r}
\end{array}\right.
\end{align*}
\end{theorem}
This result extends those from \cite{Lelarge12}, where the value of the threshold $\tau^*_{h,k,l,1}$ was computed.

\subsection{Distributed content delivery network}\label{section: content delivery network}
Consider a content delivery network (CDN) in which service can be given either from a powerful but costly data center, or from a large number of small, inexpensive servers. Content requests are then served if possible by the small servers and otherwise redirected to the datacenter.
One is then interested in determining the fraction of load that can be absorbed by the small servers. A natural asymptotic to consider is that of large number $m$ of small servers with fixed storage and service capacity and large collection $n$ of content items.


The precise model we consider follows the statistical assumptions from \cite{Leconte12}. It is described by a bipartite graph $G=(A\cup B,E)$, where $A$ is the set of servers and $B$ the set of contents, $|A|\sim|B|\tau$. An edge in $E$ between a server $s$ in $A$ and a content $c$ in $B$ indicates that server $s$ stores a copy of content $c$ and is thus able to serve requests for it.

An assignement of servers to requests corresponds exactly to an allocation of $G$ provided the vertex-constraint at server $s$ is its upload capacity, the vertex-constraint at content $c$ is its number of requests $\omega_c$, and the edge-constraint is $\infty$. Thus, $M(G)$ is the maximum number of requests absorbed by the small servers. Assuming $\Phi^A$ is the distribution of storage and upload capacity of the servers and $\Phi^B$  the distribution of number of replicas and requests of the contents, then $\Mcal(\Phi^A,\Phi^B)$ computed from Theorem~\ref{th: maximum allocation for bipartite Galton-Watson limits} is the asymptotic maximum load absorbed by the servers (in number of requests per server). This represents a generalization of the results in \cite{Leconte12} which handled only servers with unit service capacity, while our result applies to any capacity distribution with finite mean.

Furthermore, the addition of edge capacities also allows us to model more complex cases. Suppose that all contents may have unequal sizes, say the size of a randomly chosen content is a random variable $L$, and that each content is fragmented into segments of constant unit size. The storage and upload capacity of the servers is then measured in terms of size rather than number of contents, and the servers now choose which content and also which segment they store. 

Assume further that when a server chooses to cache a segment from content $c$, instead of storing the raw segment it instead stores a random linear combination of all the $l_c$ segments corresponding to content $c$. Then, when a user requests content $c$ it needs only download a coded segment from any $l_c$ servers storing segments from $c$, as any $l_c$ coded segments are sufficient to recover the content $c$. An assignement of servers to requests still corresponds to an allocation of $G$, with the vertex-constraints at servers unchanged, the vertex-constraints at content $c$ equal to $\omega_c l_c$ and the edge-constraints linked to a content $c$ equal to $\omega_c$. Indeed a given encoded segment can be used only once per request of the corresponding content. Then, letting $\Phi^A$ and $\Phi^B$ be the appropriate joint laws, $\Mcal(\Phi^A,\Phi^B)$ is the asymptotic maximum absorbed load (in number of fragments per server).

One could then follow the same path as in \cite{Leconte12} and determine the replication ratios of  contents based on a priori knowledge about  their number of requests so as to maximize the load asymptotically absorbed by the server pool; this is beyond the cope of the present paper.

\section{Main Proof Elements}\label{section: overview}
We start with a high level description of this section.
The proof strategy uses a detour, by introducing a finite parameter $\lambda>0$ playing the role of an inverse temperature. For a given finite graph $G$, a Gibbs distribution $\mu^{\lambda}_G$ is defined on edge occupancy parameters $\bx$ (Section \ref{sec:gibbs}) such that an average under $\mu_G^{\lambda}$ approaches the quantity of interest $M(G)/|V|$ as $\lambda$ tends to infinity. Instead of considering directly the limit of this parameter over a series of converging graphs $G_n$, we take an indirect route, changing the order of limits over $\lambda$ and $n$. 

We thus first determine for fixed $\lambda$ the asymptotics in $n$ of averages under $\mu^{\lambda}_{G_n}$. This is where BP comes into play. We characterize the behaviour of BP associated with $\mu^{\lambda}_G$ on finite $G$ (Section~\ref{sec:msg}), establishing its convergence to a unique fixed point thanks to structural properties of monotonicity for the upshifted likelihood ratio order, and of log-concavity of messages (Sections~\ref{sec:local} and \ref{sec:finite}). This allows to show that limits over $n$ of averages under $\mu^{\lambda}_{G_n}$ are characterized by fixed point relations {\`a la} BP. Taking limits over $\lambda\to \infty$, one derives from these fixed points the RDEs appearing in the statement of Theorem 1. It then remains to justify interchange of limits in $\lambda$ and $n$. These last three steps are handled similarly to \cite{Lelarge12} (see appendix).

Before we proceed we introduce some necessary notation. Letters or symbols in bold such as $\bx$ denote collections of objects $(x_i)_{i\in I}$ for some set $I$. For a subset $S$ of $I$, $\bx_S$ is the sub-collection $(x_i)_{i\in S}$ and $|\bx_S|:=\sum_{i\in S}x_i$ is the $L_1$-norm of $\bx_S$. Inequalities between collections of items should be understood componentwise, thus $\bx\leq\bc$ means $x_i\leq c_i$ for all $i\in I$. 
For distributions $m_i$, we let $\bm_S(\bx):=\prod_{i\in S}m_i(x_i)$. When summing such terms as in $\sum_{\bx\in\N^S:|\bx|\leq b,\:\bx\leq\bc}\bm_S(\bx)$, we shall omit the constraint $\bx\in\N^S$. Similarly, we let $\ast_S\bm=\ast_{i\in S}m_i$.

\subsection{Gibbs measure}\label{sec:gibbs}
Let $G=(V,E)$ be a finite graph, with collections of vertex- and edge-constraints $\bb=(b_v)_{v\in V}$ and $\bc=(c_e)_{e\in E}$). The Gibbs measure at temperature parameter $\lambda\in\R_+$ on the set of all vectors in $\N^E$ is then defined, for $\bx\in\N^E$, as
\begin{align*}
\mu_G^\lambda(\bx)
=\frac{1}{Z_G(\lambda)}\lambda^{|\bx|}\ind(\bx\text{ allocation of }G)
=\frac{1}{Z_G(\lambda)}\lambda^{|\bx|}\prod_{v\in V}\ind(\sum_{e\in\d v}x_e\leq b_v)\prod_{e\in E}\ind(x_e\leq c_e),
\end{align*}
where $Z_G(\lambda)$ is a normalization factor.

When $\lambda\to\infty$, $\mu_G^\lambda$ tends to the uniform probability measure on the set of all allocations of $G$ of maximum size. Thus, $\lim_{\lambda\to\infty}\mu_G(|\bX|)=M(G)$, where $\mu_G^\lambda(|\bX|)$ is the expected size of a random allocation $\bX$ drawn according to $\mu_G^\lambda$. Hence, we can compute $M(G)/|V|$ as follows:
\begin{eqnarray}
\frac{M(G)}{|V|}
=\lim_{\lambda\to\infty}\mu_G^\lambda\left(\sum_{v\in V}\frac{1}{|V|}\frac{\sum_{e\in\d v}X_e}{2}\right)
=\frac{1}{2}\lim_{\lambda\to\infty}\EE\left[\mu_G^\lambda\left(\sum_{e\in\d R}X_e\right)\right],\label{eqn: why marginals?}
\end{eqnarray}
where $R$ is a root-vertex chosen uniformly at random among all vertices in $V$, and the first expectation is with respect to the choice of $R$.

\subsection{Associated BP message passing}\label{sec:msg} 
We introduce the set $\orE$ of directed edges of $G$ comprising two directed edges $\overrightarrow{uv}$ and $\overrightarrow{vu}$ for each undirected edge $uv\in E$. We also define $\ordv$ as the set of edges directed towards vertex $v\in V$, $\oldv$ as the set of  edges directed  outwards from $v$, and $\orde:=(\overrightarrow{wv})_{w\in\d v\setminus u}$ if $\ore$ is the directed edge $\overrightarrow{vu}$.

An allocation puts an integer weight on each edge of the graph. Accordingly the messages to be sent along each edge are distributions over the integers. We let $\Pcal$ be the set of all probability distributions on integers with bounded support, i.e.
\begin{align*}
\Pcal=\left\{p\in[0,1]^\N;\sum_{i\in\N}p(i)=1\text{ and }\exists k\in\N\text{ such that }p(i)=0,\forall i>k\right\},
\end{align*}
and $\tPcal$ the set of distributions in $\Pcal$ whose support is an interval containing $0$. 

A message on directed edge $\ore$  with capacity $c_e$ is a distribution in $\Pcal$ with support in $\{0,\ldots,c_e\}$.
The message to send on edge $\ore$ outgoing from vertex $v$ is computed from the messages incoming to $v$ on the other edges via
\begin{align*}
\Rcal_\ore^{(\lambda)}[\bm](x):=\frac{\lambda^x\ind(x\leq c_{vu})\sum_{|\by|\leq b_v-x}\bm_{\d\overrightarrow{vu}}^\lambda(\by)}{\sum_{t\leq c_{vu}}\lambda^t\sum_{|\by|\leq b_v-t}\bm_{\d\overrightarrow{vu}}^\lambda(\by)},
\end{align*}
where we introduced the operator $\Rcal_\ore^{(\lambda)}:\tPcal^\orde\to\tPcal$.
For notational convenience, we write $\Rcal_\ore^{(\lambda)}[\bm]$ instead of $\Rcal_\ore^{(\lambda)}[\bm_\orde]$. We also introduce $\Rcal_\ore$ for $\Rcal_\ore^{(1)}$. The two operators are linked via the relationship
\begin{align*}
\Rcal_\ore^{(\lambda)}[\bm](x)=\frac{\lambda^x\Rcal_\ore[\bm](x)}{\sum_{t\ge 0}\lambda^t\Rcal_\ore[\bm](t)}.
\end{align*}

We also define an operator $\Dcal_v:\tPcal^\ordv\to\R^+$ meant to approximate the average occupancy at a vertex $v$ under $\mu_G^\lambda$ from the messages incoming to $v$:
\begin{align*}
\Dcal_v[\bm]=\frac{\sum_{|\bx|\leq b_v}|\bx|\bm_\ordv(\bx)}{\sum_{|\bx|\leq b_v}\bm_\ordv(\bx)}.
\end{align*}

Finally we denote by $\Rcal_G^{(\lambda)}$ the operator that performs the action of all the $\Rcal_\ore^{(\lambda)}$ for all $\ore$ simultaneously, i.e. $\Rcal_G^{(\lambda)}[\bm]=\left(\Rcal_\ore^{(\lambda)}[\bm]\right)_{\ore\in\orE}$ (the same type of notation will be used for other operators). It is well known that belief propagation converges and is exact on finite trees \cite{Mezard09}:
\begin{proposition}
In a finite tree $G$, the fixed point equation $\bm=\Rcal_G^{(\lambda)}[\bm]$ admits a unique solution $\bm^{(\lambda)}\in\tPcal^\orE$, and it satisfies for every vertex $v$:
\begin{align*}
\mu_G^\lambda\left(\sum_{e\in\d v}X_e\right)=\Dcal_v[\bm^{(\lambda)}].
\end{align*}
\end{proposition}
However, to be able to take the limit as the temperature parameter $\lambda$ goes to infinity as well as to deal with cases when $G$ is not a tree anymore, we need to study further the operators $\Rcal_\ore$ and $\Dcal_v$, which we term the \emph{local} operators.

\subsection{Structural properties of local operators}\label{sec:local} 
In this section, we focus on the one-hop neighborhood of a vertex $v$ of a graph $G$, i.e. on vertex $v$ and its set $\d v$ of adjacent edges. We thus only consider the directed edges in $\ordv\cup\oldv$. We let $b_v$ be the vertex-constraint at $v$ and $\bc=(c_e)_{e\in\d v}$ be the collection of the edge-constraints on the edges in $\d v$.

Among the many stochastic orders studied for comparing distributions (see e.g. \cite{Muller02}), the one adapted to the structure of operators $\Rcal_\ore$ and $\Dcal_v$ is the so-called \emph{upshifted likelihood-ratio} stochastic order (abbreviated $\lru$). For two distributions $m$ and $m'$ in $\Pcal$, we say that $m$ is smaller than $m'$ (for the $\lru$ stochastic order) and we write
\begin{align*}
m\leq_\lru m'\text{ if }m(i+k+l)m'(i)\leq m(i+l)m'(i+k),\forall i,k,l\in\N.
\end{align*}

In particular, if $m$ and $m'$ have the same interval as support, we have $m\leq_\lru m'\Leftrightarrow \frac{m(i+1)}{m(i)}\leq\frac{m'(i+1)}{m'(i)}$, for all $i$ for which the denominators are non-zero. 

We shall also need the following definition. A distribution $(p_j)_{j\ge 0}$ is {\bf log-concave} if its support is an interval and $p_i p_{i+2}\leq p_{i+1}^2$, for all $i\in\N$.  This property has strong ties with the $\lru$-order. In particular one can note that $p$ is  log-concave if and only if $p\leq_\lru p$. We let $\Pcallc\subset\Pcal$ be the set of all log-concave distributions over integers with finite support, and $\tPcallc=\tPcal\cap\Pcallc$:
\begin{align*}
\Pcallc=\left\{p\in[0,1]^\N;\sum_{i\in\N}p(i)=1,\:p\text{ is log-concave, and }\exists k\in\N\text{ such that }p(i)=0,\forall i>k\right\}.
\end{align*}
The key result of this Section is then the following:
\begin{proposition}[Monotonicity of the local operators for the $\lru$-order]\label{prop: monotonicity of local operators}
The operator $\Rcal_\ore^{\lambda}$ is non-increasing; furthermore, if the inputs of $\Rcal_\ore^{\lambda}$ are log-concave, then the output is also log-concave. The operator $\Dcal_v^{\lambda}$ is non-decreasing, and strictly increasing if all its inputs are log-concave with $0$ in their support.
\end{proposition}
The proof will rely on the following lemma from \cite{Shanthikumar86} establishing stablity of $\lru$-order w.r.t. convolution:
\begin{lemma}\label{lemma: composition with lru-order}
For a set $\overrightarrow S$ of directed edges, if $\bm_{\overrightarrow S}^1\leq_\lru\bm_{\overrightarrow S}^2$ in $\Pcal^{\overrightarrow S}$, then $\ast_{\overrightarrow S}\bm^1\leq_\lru\ast_{\overrightarrow S}\bm^2$.
\end{lemma}
We shall also need the following notions:
\begin{itemize}
\item the \emph{reweighting} of a vector $m$ by a vector $p$ is defined by $m\centerdot p(x):=\frac{m(x)p(x)}{\sum_{y\in\N}m(y)p(y)}$ for $x\in\N$, for $p$ and $m$ with non-disjoint supports and $|p|<\infty$ or $|m|<\infty$. If $p$ or $m$ is in $\Pcal$, then $m\centerdot p\in\Pcal$. Note that $\Rcal_\ore^{(\lambda)}[\bm]=\lambda^\N\centerdot\Rcal_\ore[\bm]$, where $\lambda^\N=(\lambda^x)_{x\in\N}$.
\item the \emph{shifted reversal} of a vector $p$ is defined by $p^R(x)=p(b_v-x)\ind(x\leq b_v)$ for $x\in\N$; if $p\in\Pcal$ and its support is included in $[0,b_v]$, then $p^R\in\Pcal$ as well. 
\end{itemize}

It is straightforward to check that
\begin{lemma}\label{lemma: reweighting and shifted reversal with lru-order}
Reweighting preserves the $\lru$-order; shifted reversal reverses the $\lru$-order.
\end{lemma}

Note that by the previous lemma it suffices  to prove the results of Proposition~\ref{prop: monotonicity of local operators} for $\Rcal_\ore$ and they will then extend to $\Rcal_\ore^{(\lambda)}$. For space reasons, we prove here only the part of the statement concerning $\Rcal_\ore$, and only for inputs in $\tPcal$. The rest of the proof is deferred to the appendix.
\begin{proof}
Let $\ore$ be an edge outgoing from vertex $v$, and $\bm^1_\orde,\bm^2_\orde\in\tPcal^\orde$ such that $\bm^1_\orde\leq_\lru\bm^2_\orde$. Let $\delta_{[0,b_v]}(x)=\ind(0\leq x\leq b_v)$; we have $\delta_{[0,b_v]}\ast_\orde\bm^i(x)=\sum_{x-b_v\leq|\by|\leq x}\bm^i_\orde(\by)$. $\delta_{[0,b_v]}$ is log-concave, so $\delta_{[0,b_v]}\leq_\lru\delta_{[0,b_v]}$ and Lemma~\ref{lemma: composition with lru-order} then implies $\delta_{[0,b_v]}\ast_\orde\bm^1\leq_\lru\delta_{[0,b_v]}\ast_\orde\bm^2$. Lemma~\ref{lemma: reweighting and shifted reversal with lru-order} then says $\left(\delta_{[0,b_v]}\ast_\orde\bm^1\right)^R\geq_\lru\left(\delta_{[0,b_v]}\ast_\orde\bm^2\right)^R$. It is easy to check that
\begin{eqnarray}\label{eqn: expression Rcal}
\Rcal_\ore[\bm^i]=\delta_{[0,c_e]}\centerdot\left(\delta_{[0,b_v]}\ast_\orde\bm^i\right)^R;
\end{eqnarray} and as $\left(\delta_{[0,b_v]}\ast_\orde\bm^i\right)^R(0)>0$ Lemma~\ref{lemma: reweighting and shifted reversal with lru-order} again implies that $\Rcal_\ore[\bm^1]\geq_\lru\Rcal_\ore[\bm^2]$.

If now $\bm_\orde\in\tPcallc^\orde$, then $\bm_\orde\leq_\lru\bm_\orde$ and $\Rcal_\ore[\bm]\geq_\lru\Rcal_\ore[\bm]$, hence $\Rcal_\ore[\bm]\in\tPcallc$.
\end{proof}

To pave the way for the analysis of the limit $\lambda\to\infty$, we distinguish between two collections of messages $\bm_\ordv$ and $\bn_\ordv$ in $\tPcal^\ordv$ and introduce additional operators. 
For an edge $\ore$ outgoing from $v$ we define the operator $\Qcal_\ore^{(\lambda)}:\tPcal^\orde\to\tPcal$ by $\Qcal_\ore^{(\lambda)}[\bn]=\Rcal_\ore^{(\lambda)}[\lambda^\N\centerdot\bn]$,  where $\lambda^\N\centerdot\bn=\left(\lambda^\N\centerdot n_\ore\right)_{\ore\in\orE}$. As reweighting preserves the $\lru$-order, the operator $\Qcal_\ore^{(\lambda)}$ is non-increasing. It also verifies the following useful monotonicity property with respect to $\lambda$, proven in the appendix:
\begin{proposition}[Monotonicity in $\lambda$]\label{prop: monotonicity in lambda}
For $\bn_\orde\in\tPcal^\orde$, the mapping $\lambda\mapsto\Qcal_\ore^{(\lambda)}[\bn]$ is non-decreasing.
\end{proposition}

As $\lambda\to\infty$, limiting messages may not have $0$ in their support. We thus define $\alpha_\ore$ as the infimum of the support of $m_\ore\in\Pcal$, i.e. $\alpha_\ore=\min\{x\in\N:\:m_\ore(x)>0\}$, and $\beta_\ore$ as the supremum of the support of $n_\ore\in\tPcal$, i.e. $\beta_\ore=\max\{x\in\N:\:n_\ore(x)>0\}$. When there may be confusion, we will write $\alpha(m_\ore)$ and $\beta(m_\ore)$ for the infimum and the supremum of the support of $m_\ore$. We also extend the definition of the local operators given previously so that they allow inputs with arbitrary supports in $\N$: for an edge $\ore$ outgoing from vertex $v$, we define $\Rcal_\ore:\Pcal^\orde\to\tPcal$, $\Dcal_v:\Pcal^\ordv\to\R^+$, $\Qcal_\ore:\tPcal^\orde\to\Pcal$ and $\Scal_\ore:\N^\orde\to\N$ as
\begin{eqnarray}
\Rcal_\ore[\bm](x)&=&\left\{\begin{array}{ll}\frac{\ind(x\leq c_e)\sum_{|\by|\leq b_v-x}\bm_\orde(\by)}{\sum_{t\leq c_e}\sum_{|\by|\leq b_v-t}\bm_\orde(\by)}\\\delta_0(x)\end{array}\begin{array}{ll}\text{ if }|\balpha_\orde|\leq b_v\\\text{ otherwise}\end{array}\right.\label{eqn: def Rcal}\\
\Dcal_v[\bm]&=&\left\{\begin{array}{ll}\frac{\sum_{|\bx|\leq b_v}|\bx|\bm_\ordv(\bx)}{\sum_{|\bx|\leq b_v}\bm_\ordv(\bx)}\\b_v\end{array}\begin{array}{ll}\text{ if }|\balpha_\ordv|\leq b_v\\\text{ otherwise}\end{array}\right.\label{eqn: def Dcal}\\
\Qcal_\ore[\bn](x)&=&\left\{\begin{array}{ll}\frac{\ind(x\leq c_e)\sum_{|\by|=b_v-x}\bn_\orde(\by)}{\sum_{t\leq c_e}\sum_{|\by|=b_v-t}\bn_\orde(\by)}\\\delta_{c_e}(x)\end{array}\begin{array}{ll}\text{ if }|\bbeta_\orde|\geq b_v-c_e\\\text{ otherwise}\end{array}\right.\label{eqn: def Qcal}\\
\Scal_\ore(\bx)&=&\left[b_v-|\bx_\orde|\right]_0^{c_e}.
\end{eqnarray}
Note that the support of $\Rcal_\ore[\bm]$ is $\{0,\ldots,\Scal_\ore(\balpha)\}$ and that of $\Qcal_\ore[\bn]$ is $\{\Scal_\ore(\bbeta),\ldots,c_e\}$. The following result is established in the appendix:

\begin{proposition}[Continuity for log-concave inputs and limiting operators]\label{prop: continuity for log-concave inputs and limiting operator}
The operators $\Rcal_\ore$ and $\Dcal_v$ given by equations~(\ref{eqn: def Rcal}),(\ref{eqn: def Dcal}) are continuous for the $L_1$ norm  for inputs in $\tPcallc$. Also, $\Qcal_\ore$ defined in equation~(\ref{eqn: def Qcal}) satisfies $\Qcal_\ore[\bn]=\lim\uparrow_{\lambda\to\infty}\Qcal_\ore^{(\lambda)}[\bn]$ for any $\bn_\orde\in\tPcal^\orde$.
\end{proposition}

It follows naturally that $\Qcal_\ore$ is non-increasing. Moreover, we can extend the results of Proposition~\ref{prop: monotonicity of local operators} to the extended operators, i.e. $\Rcal_\ore$ is still non-increasing and $\Dcal_v$ non-decreasing.

\subsection{Convergence of BP on finite graphs}\label{sec:finite} 
The main result of this section is the following
\begin{proposition}[Convergence of BP to a unique fixed point]\label{prop: unique fixed-point in finite graphs}
Synchronous BP message updates according to $\bm^{t+1}=\Rcal_G^{(\lambda)}[\bm^t]$ for $t\geq 0$ converge to the unique solution $\bm^{(\lambda)}$ of the fixed point equation $\bm=\Rcal_G^{(\lambda)}[\bm]$.
\end{proposition}
\begin{proof}
For all $\ore\in\orE$ initialize the message on $\ore$ at $m_\ore^0=\delta_0\in\tPcallc$. As $\Rcal_G^{(\lambda)}$ is non-increasing and $\delta_0$ is a smallest element for the $\lru$ order, it can readily be shown that the following inequalities hold for all $t\ge 0$:
$$
\bm^{2t}\le_{\lru} \bm^{2t+2}\le_{\lru}\bm^{2t+3}\le_{\lru}\bm^{2t+1}.
$$ 
In other words the two series $(\bm^{2t})_{t\geq0}$ and $(\bm^{2t+1})_{t\geq 0}$ are {\em adjacent} and hence converge to respective limits $\bm^-$, $\bm^+$ such that $\bm^-\leq_{\lru} \bm^+$. Continuity of $\Rcal_G^{(\lambda)}$ further guarantees that $\bm^+=\Rcal_G^{(\lambda)}(\bm^-)$ and $\bm^-=\Rcal_G^{(\lambda)}(\bm^+)$. Moreover, considering any other sequence of vectors of messages $(\bm'^t)_{t\geq 0}$ with an arbitrary initialization, since $\bm^0\leq_\lru\bm'^0$, monotonicity of $\Rcal_G^{(\lambda)}$ ensures that for all $t\geq 0$, one has
$$
\bm^{2t}\leq_\lru \bm'^{2t},\bm'^{2t+1}\leq_\lru \bm^{2t+1}.
$$
The result will then follow if we can show that $\bm^+=\bm^-$. 

We establish this by exploiting the fact that $\Dcal_v$ is strictly increasing for inputs in $\tPcallc$.
As $\bm^{-}\leq_\lru\bm^{+}$ and $\Dcal_v$ is non-decreasing for the $\lru$-order for all $v\in V$, it follows $\Dcal_v[\bm^{-}]\leq\Dcal_v[\bm^{+}]$ for all $v\in V$. Then, summing over all vertices of $G$, we get
\begin{align*}
\sum_{v\in V}\Dcal_v[\bm^{-}]&=\sum_{v\in V}\sum_{u\sim v}\frac{\sum_{x\in\N}xm_{\overrightarrow{uv}}^{-}(x)\Rcal_{\overrightarrow{vu}}[\bm^{-}](x)}{\sum_{x\in\N}m_{\overrightarrow{uv}}^{-}(x)\Rcal_{\overrightarrow{vu}}[\bm^{-}](x)}
=\sum_{v\in V}\sum_{u\sim v}\frac{\sum_{x\in\N}x\Rcal_{\overrightarrow{uv}}[\bm^{+}](x)m_{\overrightarrow{vu}}^{+}(x)}{\sum_{x\in\N}\Rcal_{\overrightarrow{uv}}[\bm^{+}](x)m_{\overrightarrow{vu}}^{+}(x)}\\
&=\sum_{u\in V}\sum_{v\sim u}\frac{\sum_{x\in\N}x\Rcal_{\overrightarrow{uv}}[\bm^{+}](x)m_{\overrightarrow{vu}}^{+}(x)}{\sum_{x\in\N}\Rcal_{\overrightarrow{uv}}[\bm^{+}](x)m_{\overrightarrow{vu}}^{+}(x)}
=\sum_{u\in V}\Dcal_u[\bm^{+}].
\end{align*}
Hence, in fact, $\Dcal_v[\bm^{-}]=\Dcal_v[\bm^{+}]$ for all $v\in V$. As $\Dcal_v$ is strictly increasing for these inputs, $\bm^-=\bm^+=\bm^{(\lambda)}$ follows.
\end{proof}
We finish this Section by stating results on the limiting behaviour of the fixed point of BP on a fixed finite graph $G$ as $\lambda\to\infty$. 
First, adapting the argument of Chertkov \cite{Chertkov08} which dealt only with unitary capacities, we can show
\begin{proposition}[Correctness for finite bipartite graphs]\label{prop: Chertkov}
In finite bipartite graphs,
\begin{align*}
\frac{1}{2}\lim\uparrow_{\lambda\to\infty}\sum_{v\in V}\Dcal_v[\bm^{(\lambda)}]=M(G).
\end{align*}
\end{proposition}
\begin{remark}
In view of this proposition, BP can be used as an algorithm to compute the maximum size of allocations in finite bipartite graphs, by running the algorithm at finite temperature parameter $\lambda$, computing $\Dcal_v[\bm^{(\lambda)}]$ for all $v$ from the fixed-point messages, and then letting $\lambda\to\infty$.
\end{remark}
In the non-bipartite case, the fixed-point $\bm^{(\lambda)}$ at finite $\lambda$ admits a limit $\bm^{(\infty)}$, and the value of $\sum_v\Dcal_v[\bm^{(\infty)}]$ is equal to $\sum_vF_v(\balpha^{(\infty)})$, where $F_v$ is defined in the propositions below (whose proof is in the appendix). This sum is computed from the infimum $\balpha^{(\infty)}$ of the support of $\bm^{(\infty)}$. Furthermore, $\balpha^{(\infty)}$ can also be obtained from a fixed-point equation, of which it is the solution that gives the lowest value of $\sum_vF_v$.
\begin{proposition}[Limit of $\lambda\to\infty$]\label{prop: limit of $0$ temperature}
$\bm^{(\lambda)}$ is non-decreasing in $\lambda$ for the $\lru$-order, and $\bm^{(\infty)}=\lim\uparrow_{\lambda\to\infty}\bm^{(\lambda)}\in\Pcallc^\orE$ is the minimal solution (for the $\lru$-order) of $\bm^{(\infty)}=\Qcal_G\circ\Rcal_G[\bm^{(\infty)}]$.
\end{proposition}

\begin{proposition}[BP estimate in finite graphs]\label{prop: BP estimate in finite graphs}
In a finite graph $G$, we have
\begin{align*}
\lim\uparrow_{\lambda\to\infty}\sum_{v\in V}\Dcal_v[\bm^{(\lambda)}]=\sum_{v\in V}\Dcal_v[\bm^{(\infty)}]=\sum_{v\in V}F_v(\balpha^{(\infty)})=\inf_{\balpha=\Scal_G\circ\Scal_G(\balpha)} \sum_{v\in V}F_v(\balpha),
\end{align*}
where $F_v(\balpha)=\min(b_v,|\balpha_\ordv|)+(b_v-|\balpha_\oldv|)^+$.
\end{proposition}
\begin{remark}
In a finite tree, there is only one possible value for $\alpha_\ore=\Scal_\ore\circ\Scal_\orde[\balpha]$ when $\ore$ is an edge outgoing from a leaf $v$: it is $\alpha_\ore=\min\{b_v,c_e\}$. It is then possible to compute the whole, unique fixed-point vector $\balpha=\Scal_G\circ\Scal_G(\balpha)$ in an interative manner, starting from the leaves of the tree and climbing up. This gives a simple, iterative way to compute the maximum size of allocations in finite trees, which is the natural extension of the leaf-removal algorithm for matchings.
\end{remark}

\subsection{Infinite unimodular graphs}
This section extends the results obtained so far for finite graphs to infinite graphs. As in \cite{Salez11,Lelarge12}, we use for this the framework of \cite{Aldous07}. We still denote by $G=(V,E)$ a possibly infinite graph with vertex set $V$ and undirected edge set $E$ (and directed edge set $\orE$). We always assume that the degrees are finite, i.e. the graph is locally finite. A network is a graph $G$ together with a complete separable metric space $\Xi$ called the mark space, and maps from $V$ and $\orE$ to $\Xi$. Images in $\Xi$ are called marks. A rooted network $(G,r)$ is a network with a distinguished vertex $r$ of $V$ called the root. A rooted isomorphism of rooted networks is an isomorphism of the underlying networks that takes the root of one to the root of the other. We do not distinguish between a rooted network and its isomorphism class denoted by $[G,r]$. Indeed, it is shown in \cite{Aldous07} how to define a canonical representative of a rooted isomorphism class.

Let $\Gcal_*$ denote the set of rooted isomorphism classes of rooted connected locally finite networks. Define a metric on $\Gcal_*$ by letting the distance between $[G_1,r_1]$ and $[G_2,r_2]$ be $1/(1+\delta)$ where $\delta$ is the supremum of those $d\geq0$ such that there is some rooted isomorphism of the balls of graph-distance radius $\lfloor d\rfloor$ around the roots of $G_i$ such that each pair of corresponding marks has distance less than $1/d$. $\Gcal_*$ is separable and complete in this metric \cite{Aldous07}.

Similarly to the space $\Gcal_*$, we define the space $\Gcal_{**}$ of isomorphism classes of locally finite connected networks with an ordered pair of distinguished vertices and the natural topology thereon.
\begin{definition}
Let $\rho$ be a probability measure on $\Gcal_*$. We call $\rho$ unimodular if it obeys the Mass-Transport Principle (MTP): for Borel $f:\Gcal_{**}\to[0,\infty]$, we have
\begin{align*}
\int\sum_{v\in V}f(G,r,v)d\rho([G,r])=\int\sum_{v\in V}f(G,v,r)d\rho([G,r])
\end{align*}
\end{definition}

Let $\Ucal$ denote the set of unimodular Borel probability measures on $\Gcal_*$. For $\rho\in\Ucal$, we write $\overline b(\rho)$ for the expectation of the capacity constraint of the root with respect to $\rho$. Our first result (proved in the appendix) is that the BP updates admit a unique fixed-point at finite temperature parameter $\lambda$:
\begin{proposition}\label{prop: unique fixed-point in unimodular graphs}
Let $\rho\in\Ucal$ with $\overline b(\rho)<\infty$. Then, the fixed point equation $\bm=\Rcal^{(\lambda)}[\bm]$ admits a unique solution $\balpha^{(\lambda)}$ for any $\lambda\in\R^+$ for $\rho$-almost every marked graph $G$.
\end{proposition}

The proof differs from that in the finite graph case in that we cannot sum $\Dcal_v$ over all the vertices $v\in V$ anymore. Instead, we use the MTP for $f(G,r,v)=\frac{\sum_{x\in\N}xm_{\overrightarrow{vr}}^{(\lambda,-)}(x)\Rcal_{\overrightarrow{rv}}[\bm^{(\lambda,-)}](x)}{\sum_{x\in\N}m_{\overrightarrow{vr}}^{(\lambda,-)}(x)\Rcal_{\overrightarrow{rv}}[\bm^{(\lambda,-)}](x)}$.

The rest of the reasoning goes as in the finite graph case and the proofs can be found in the appendix (using the MTP again, instead of summing over all directed edges): Proposition~\ref{prop: limit of $0$ temperature} is still valid and the following proposition is analogous to Proposition~\ref{prop: BP estimate in finite graphs}:
\begin{proposition}[BP estimate in unimodular random graphs]\label{prop: BP estimate in unimodular random graphs}
Let $\rho\in\Ucal$ with $\overline b(\rho)<\infty$,
\begin{align*}
\lim\uparrow_{\lambda\to\infty}\int\Dcal_r[\bm^{(\lambda)}]d\rho([G,r])&=\int\Dcal_r[\bm^{(\infty)}]d\rho([G,r])=\int F_r(\balpha^{(\infty)})d\rho([G,r])\\
&=\inf_{\balpha=\Scal_G\circ\Scal_G(\balpha)}\int F_r(\balpha)d\rho([G,r]),
\end{align*}
where $F_v(\balpha)=\min(b_v,|\balpha_\ordv|)+(b_v-|\balpha_\oldv|)^+$.
\end{proposition}

\subsection{From finite graphs to unimodular trees}

Once Proposition \ref{prop: BP estimate in unimodular random graphs}
holds, the end of the proof for sequences of (sparse) random graphs is
quite systematic and follows the same steps as in \cite{bls11},
\cite{Salez11} and \cite{Lelarge12}. We first need to show that we can
invert the limits in $n$ and $\lambda$ (see Proposition 6 in
\cite{Lelarge12}):
\begin{proposition}[Asymptotic correctness for large, sparse random graphs]\label{prop: asymptotic correctness for unimodular random trees}
Let $G_n=(V_n,E_n)_n$ be a sequence of finite marked graphs with random weak limit $\rho$ concentrated on unimodular trees, with $\overline b(\rho)<\infty$. Then,
\begin{align*}
\lim_{n\to\infty}\frac{2M_n}{|V_n|}=\int\Dcal_r[\bm^{(\infty)}]d\rho([G,r])=\inf_{\balpha=\Scal_G\circ\Scal_G(\balpha)}\int F_r(\balpha)d\rho([G,r]).
\end{align*}
\end{proposition}

The second step uses the Markovian nature of
the limiting Galton-Watson tree to simplify the infinite recursions
$\balpha=\Scal_G\circ\Scal_G(\balpha)$ into recursive distributional
equations as described in Theorem \ref{th: maximum allocation for bipartite Galton-Watson limits}.
Finally, the fact that the sequence of graphs considered in the
introduction converges locally weakly to unimodular Galton-Watson
trees follows from standard results in the random graphs literature
(see \cite{Kim08} for random hypergraphs or \cite{demmon10} for graphs
with fixed degree sequence).

\bibliographystyle{abbrv}

\begin{thebibliography}{10}

\bibitem{Aldous07}
D.~Aldous and R.~Lyons.
\newblock Processes on unimodular random networks.
\newblock {\em Electron. J. Probab.}, 12:no. 54, 1454--1508, 2007.

\bibitem{aldste}
D.~Aldous and J.~M. Steele.
\newblock The objective method: probabilistic combinatorial optimization and
  local weak convergence.
\newblock In {\em Probability on discrete structures}, volume 110 of {\em
  Encyclopaedia Math. Sci.}, pages 1--72. Springer, Berlin, 2004.

\bibitem{Aldous05}
J.~Aldous and A.~Bandyopadhyay.
\newblock A survey of max-type recursive distributional equations.
\newblock {\em Annals of Applied Probability 15 (2005}, 15:1047--1110, 2005.

\bibitem{bbcz07}
M.~Bayati, C.~Borgs, J.~T. Chayes, and R.~Zecchina.
\newblock Belief-propagation for weighted b-matchings on arbitrary graphs and
  its relation to linear programs with integer solutions.
\newblock {\em CoRR}, abs/0709.1190, 2007.

\bibitem{bls11}
C.~Bordenave, M.~Lelarge, and J.~Salez.
\newblock Matchings on infinite graphs.
\newblock {\em Arxiv preprint arXiv:1102.0712}, 2011.

\bibitem{Cain07}
J.~A. Cain, P.~Sanders, and N.~Wormald.
\newblock The random graph threshold for k-orientiability and a fast algorithm
  for optimal multiple-choice allocation.
\newblock In {\em Proceedings of the eighteenth annual ACM-SIAM symposium on
  Discrete algorithms}, SODA '07, pages 469--476, Philadelphia, PA, USA, 2007.
  Society for Industrial and Applied Mathematics.

\bibitem{Chertkov08}
M.~Chertkov.
\newblock Exactness of belief propagation for some graphical models with loops.
\newblock {\em CoRR}, abs/0801.0341, 2008.

\bibitem{demmon10}
A.~Dembo and A.~Montanari.
\newblock Gibbs measures and phase transitions on sparse random graphs.
\newblock {\em Braz. J. Probab. Stat.}, 24(2):137--211, 2010.

\bibitem{Dietzfelbinger10}
M.~Dietzfelbinger, A.~Goerdt, M.~Mitzenmacher, A.~Montanari, R.~Pagh, and
  M.~Rink.
\newblock Tight thresholds for cuckoo hashing via xorsat.
\newblock In {\em Proceedings of the 37th international colloquium conference
  on Automata, languages and programming}, ICALP'10, pages 213--225, Berlin,
  Heidelberg, 2010. Springer-Verlag.

\bibitem{Dietzfelbinger05}
M.~Dietzfelbinger and C.~Weidling.
\newblock Balanced allocation and dictionaries with tightly packed constant
  size bins.
\newblock In {\em Proceedings of the 32nd international conference on Automata,
  Languages and Programming}, ICALP'05, pages 166--178, Berlin, Heidelberg,
  2005. Springer-Verlag.

\bibitem{Fernholz07}
D.~Fernholz and V.~Ramachandran.
\newblock The k-orientability thresholds for gn, p.
\newblock In {\em Proceedings of the eighteenth annual ACM-SIAM symposium on
  Discrete algorithms}, SODA '07, pages 459--468, Philadelphia, PA, USA, 2007.
  Society for Industrial and Applied Mathematics.

\bibitem{Fountoulakis11}
N.~Fountoulakis, M.~Khosla, and K.~Panagiotou.
\newblock The multiple-orientability thresholds for random hypergraphs.
\newblock In {\em Proceedings of the Twenty-Second Annual ACM-SIAM Symposium on
  Discrete Algorithms}, SODA '11, pages 1222--1236. SIAM, 2011.

\bibitem{Frieze09}
A.~M. Frieze and P.~Melsted.
\newblock Maximum matchings in random bipartite graphs and the space
  utilization of cuckoo hashtables.
\newblock {\em CoRR}, abs/0910.5535, 2009.

\bibitem{Gao10}
P.~Gao and N.~C. Wormald.
\newblock Load balancing and orientability thresholds for random hypergraphs.
\newblock In {\em Proceedings of the 42nd ACM symposium on Theory of
  computing}, STOC '10, pages 97--104, New York, NY, USA, 2010. ACM.

\bibitem{Kim08}
J.~H. {Kim}.
\newblock {Poisson Cloning Model for Random Graphs}.
\newblock {\em ArXiv e-prints}, May 2008.

\bibitem{kmrsz}
F.~Krz{\c{a}}ka{\l}a, A.~Montanari, F.~Ricci-Tersenghi, G.~Semerjian, and
  L.~Zdeborov{\'a}.
\newblock Gibbs states and the set of solutions of random constraint
  satisfaction problems.
\newblock {\em Proc. Natl. Acad. Sci. USA}, 104(25):10318--10323 (electronic),
  2007.

\bibitem{Leconte12}
M.~Leconte, M.~Lelarge, and L.~Massouli{\'e}.
\newblock Bipartite graph structures for efficient balancing of heterogeneous
  loads.
\newblock In {\em Proceedings of the 12th ACM SIGMETRICS/PERFORMANCE joint
  international conference on Measurement and Modeling of Computer Systems},
  SIGMETRICS '12, pages 41--52, New York, NY, USA, 2012. ACM.

\bibitem{Lelarge12}
M.~Lelarge.
\newblock A new approach to the orientation of random hypergraphs.
\newblock In {\em Proceedings of the Twenty-Third Annual ACM-SIAM Symposium on
  Discrete Algorithms}, SODA '12, pages 251--264. SIAM, 2012.

\bibitem{lp09}
L.~Lov{\'a}sz and M.~D. Plummer.
\newblock {\em Matching theory}.
\newblock AMS Chelsea Publishing, Providence, RI, 2009.
\newblock Corrected reprint of the 1986 original [MR0859549].

\bibitem{mmw05}
E.~Maneva, E.~Mossel, and M.~J. Wainwright.
\newblock A new look at survey propagation and its generalizations.
\newblock In {\em Proceedings of the sixteenth annual ACM-SIAM symposium on
  Discrete algorithms}, SODA '05, pages 1089--1098, Philadelphia, PA, USA,
  2005. Society for Industrial and Applied Mathematics.

\bibitem{Mezard09}
M.~Mezard and A.~Montanari.
\newblock {\em Information, Physics, and Computation}.
\newblock Oxford University Press, Inc., New York, NY, USA, 2009.

\bibitem{mpv87}
M.~M{\'e}zard, G.~Parisi, and M.~A. Virasoro.
\newblock {\em Spin glass theory and beyond}, volume~9 of {\em World Scientific
  Lecture Notes in Physics}.
\newblock World Scientific Publishing Co. Inc., Teaneck, NJ, 1987.

\bibitem{Muller02}
A.~M\"uller and D.~Stoyan.
\newblock {\em Comparison Methods for Stochastic Models and Risks}.
\newblock Wiley, 2009.

\bibitem{pearl88}
J.~Pearl.
\newblock {\em Probabilistic reasoning in intelligent systems: networks of
  plausible inference}.
\newblock The Morgan Kaufmann Series in Representation and Reasoning. Morgan
  Kaufmann, San Mateo, CA, 1988.

\bibitem{ricurb08}
T.~Richardson and R.~Urbanke.
\newblock {\em Modern coding theory}.
\newblock Cambridge University Press, Cambridge, 2008.

\bibitem{Salez11}
J.~{Salez}.
\newblock {The cavity method for counting spanning subgraphs subject to local
  constraints}.
\newblock {\em ArXiv e-prints}, Mar. 2011.

\bibitem{Shanthikumar86}
J.~G. Shanthikumar and D.~D. Yao.
\newblock The preservation of likelihood ratio ordering under convolution.
\newblock {\em Stochastic Processes and their Applications}, 23(2):259--267,
  1986.

\bibitem{yfw05}
J.~Yedidia, W.~Freeman, and Y.~Weiss.
\newblock Constructing free-energy approximations and generalized belief
  propagation algorithms.
\newblock {\em Information Theory, IEEE Transactions on}, 51(7):2282 -- 2312,
  july 2005.

\end{thebibliography}

\newpage
\section{Appendix:}
\subsection{Local operators}
\begin{proposition*}[Monotonicity of the local operators for the $\lru$-order; Proposition~\ref{prop: monotonicity of local operators}]
The operator $\Rcal_\ore$ is non-increasing; furthermore, if the inputs of $\Rcal_\ore$ are log-concave, then the output is also log-concave. The operator $\Dcal_v$ is non-decreasing, and strictly increasing if all its inputs are log-concave with $0$ in their support.
\end{proposition*}
\begin{proof}
Let $\ore$ be an edge outgoing from vertex $v$, and $\bm^1_\ordv,\bm^2_\ordv\in\Pcal^\ordv$ such that $\bm^1_\ordv\leq_\lru\bm^2_\ordv$. Firstly, if $|\balpha^2_\orde|\geq b_v$, then $\Rcal_\ore[\bm^2]=\delta_0$ and automatically $\Rcal_\ore[\bm^1]\geq_\lru\delta_0=\Rcal_\ore[\bm^2]$. Then, if $|\balpha^2_\orde|\leq b_v$, we also have $|\balpha^1_\orde|\leq |\balpha^2_\orde|\leq b_v$. Let $\delta_{[0,b_v]}(x)=\ind(0\leq x\leq b_v)$ and $\theta^i_\ore=\ast_\orde\bm^i$; we have $\delta_{[0,b_v]}\ast\theta_\ore^i(x)=\sum_{x-b_v\leq|\by|\leq x}\bm^i_\orde(\by)$. $\delta_{[0,b_v]}$ is log-concave, so $\delta_{[0,b_v]}\leq_\lru\delta_{[0,b_v]}$ and Lemma~\ref{lemma: composition with lru-order} then implies $\delta_{[0,b_v]}\ast\theta^1_\ore\leq_\lru\delta_{[0,b_v]}\ast\theta^2_\ore$. Lemma~\ref{lemma: reweighting and shifted reversal with lru-order} then says $\left(\theta^1_\ore\right)^R\geq_\lru\left(\theta^2_\ore\right)^R$. It is easy to check that
\begin{align*}
\Rcal_\ore[\bm^i]=\delta_{[0,c_e]}\centerdot\left(\delta_{[0,b_v]}\ast_\orde\bm^i\right)^R;
\end{align*} and furthermore, as $\left(\delta_{[0,b_v]}\ast_\orde\bm^i\right)^R(0)>0$, Lemma~\ref{lemma: reweighting and shifted reversal with lru-order} again implies that $\Rcal_\ore[\bm^1]\geq_\lru\Rcal_\ore[\bm^2]$.

If now $\bm_\orde\in\Pcallc^\orde$, then $\bm_\orde\leq_\lru\bm_\orde$ and $\Rcal_\ore[\bm]\geq_\lru\Rcal_\ore[\bm]$, which shows $\Rcal_\ore[\bm]\in\Pcallc$.

Similarly, if $|\balpha_\ordv^2|\geq b_v$ then $\Dcal_v[\bm^2]=b_v$ and automatically $\Dcal_v[\bm^1]\leq b_v=\Dcal_v[\bm^2]$. If now $|\balpha_\ordv^2|<b_v$, we also have $|\balpha_\ordv^1|<b_v$. Lemma~\ref{lemma: composition with lru-order} shows $\theta_v^1=\ast_\ordv\bm^1\leq_\lru\theta_v^2=\ast_\ordv\bm^2$. As $|\balpha^i|\leq b_v$, Lemma~\ref{lemma: reweighting and shifted reversal with lru-order} says $\delta_{[0,b_v]}\centerdot\theta_v^1\leq_\lru\delta_{[0,b_v]}\centerdot\theta_v^2$. This implies that the mean of $\delta_{[0,b_v]}\centerdot\theta_v^1$ is no larger than that of $\delta_{[0,b_v]}\centerdot\theta_v^2$, which is exactly $\Dcal_v[\bm^1]\leq\Dcal_v[\bm^2]$.

Furthermore, if $\bm_\ordv^1<_\lru\bm_\ordv^2$ in $\Pcallc^\ordv$ and $|\balpha_\ordv^1|=|\balpha_\ordv^2|=0$, then a direct calculation will shows that $\gamma_v^1<_\lru\gamma_v^2$, which implies $\Dcal_v[\bm^1]<\Dcal_v[\bm^2]$. More precisely, fix $\ore\in\ordv$; it is sufficient to work with $m^1_{\ore'}=m^2_{\ore'}$ for all $\ore'\neq\ore$, as then the loose inequality obtained before allows to conclude. Then, there exists a minimum $i\in\N$ such that $m_\ore^1(i+1)m^2_\ore(i)<m_\ore^1(i)m^2_\ore(i+1)$. It is then immediate that $\theta_v^1(i+1)\theta_v^2(i)<\theta_v^1(i)\theta_v^2(i+1)$ as the only term differing between the two sides is the one for which we have strict inequality. This implies $\left(\theta_v^1(x)\right)_{x\leq b_v}<_\lru\left(\theta_v^2(x)\right)_{x\leq b_v}$, as we already obtained the loose inequality. For these vectors, Lemma~\ref{lemma: reweighting and shifted reversal with lru-order} tells us that reweighting by $\delta_{[0,b_v]}$ preserves the strict $\lru$-ordering, thus we have $\delta_{[0,b_v]}\centerdot\theta_v^1<_\lru\delta_{[0,b_v]}\centerdot\theta_v^2$ as claimed.
\end{proof}

\begin{proposition*}[Monotonicity in $\lambda$; Proposition~\ref{prop: monotonicity in lambda}]
For $\bn_\orde\in\tPcal^\orde$, the mapping $\lambda\mapsto\Qcal_\ore^{(\lambda)}[\bn]$ is non-decreasing.
\end{proposition*}
\begin{proof}
Let $\ore$ be an edge outgoing from $v$, and $\bn_\orde\in\tPcal^\orde$. First-of-all, the support of the vectors considered is independent of $\lambda\in\R^+$. We will use the expression of $\Rcal_e$ from equation~(\ref{eqn: expression Rcal}). It is easy to check that
\begin{align*}
\Qcal_\ore^{(\lambda)}[\bn]=\lambda^\N\centerdot\delta_{[0,c_e]}\centerdot\left(\delta_{[0,b_v]}\ast_\orde\left(\lambda^\N\centerdot\bn\right)\right)^R=\delta_{[0,c_e]}\centerdot\left(\left(\lambda^\N\right)^R\ast_\orde\bn\right)^R,
\end{align*}
which shows that $\Qcal_\ore^{(\lambda)}[\bn]$ is non-decreasing in $\lambda$ (as $R$ is the only operator used which reverses the $\lru$-order instead of preserving it, and it is applied twice to $\lambda^\N$).
\end{proof}

\begin{proposition*}[Continuity for log-concave inputs and limiting operators; Proposition~\ref{prop: continuity for log-concave inputs and limiting operator}]
The operators $\Rcal_\ore$ and $\Dcal_v$ given by equations~(\ref{eqn: def Rcal},\ref{eqn: def Dcal}) are continuous (for the $L_1$-norm) for log-concave inputs with $0$ in their support. Also, $\Qcal_\ore$ defined in equation~(\ref{eqn: def Qcal}) satisfies $\Qcal_\ore[\bn]=\lim\uparrow_{\lambda\to\infty}\Qcal_\ore^{(\lambda)}[\bn]$ for any $\bn_\orde\in\tPcal^\orde$.
\end{proposition*}
\begin{proof}
Consider a sequence $(\bm_\ordv^{(k)})_{k\in\N}$ in $\tPcallc^\ordv$ converging towards $\bm_\ordv$, which thus belongs to $\Pcallc^\ordv$. Let $\ore$ be an edge outgoing from $v$:
\begin{align*}
\Rcal_\ore[\bm^{(k)}](x)=\frac{\ind(x\leq c_e)\sum_{|\by|\leq b_v-x}\bm_\orde^{(k)}(\by)}{\sum_{t\leq c_e}\sum_{|\by|\leq b_v-t}\bm_\orde^{(k)}(\by)}
\end{align*}
If $|\balpha_\orde|\leq b_v$, then the expression above is clearly continous. If $|\balpha_\orde|\geq b$, then the numerator is equivalent as $k\to\infty$ to $\sum_{|\by|=b_v}\bm_\orde^{(k)}(\by)$ and the numerator to $\sum_{|\by|=b_v-x}\bm_\orde^{(k)}(\by)$, for $x\leq c_e$. Then, $$\lim_{k\to\infty}\Rcal_\ore[\bm^{(k)}](x)=\delta_0(x).$$

A similar reasoning applies to $\Dcal_v$:
\begin{align*}
\Dcal_v[\bm^{(k)}]=\frac{\sum_{|\bx|\leq b_v}|\bx|\bm_\ordv^{(k)}(\bx)}{\sum_{|\bx|\leq b_v}\bm_\ordv^{(k)}(\bx)}
\end{align*}
If $|\balpha_\ordv|\leq b_v$ then the expression above is clearly continous. If $|\balpha_\ordv|\geq b_v$, then the denominator is equivalent to $\sum_{|\bx|=b_v}b_v\bm_\ordv^{(k)}(\bx)$ and the numerator to $\sum_{|\bx|=b_v}\bm_\ordv^{(k)}(\bx)$. So, $$\lim_{k\to\infty}\Dcal_v[\bm^{(k)}]=b_v.$$

We now turn to the operator $\Qcal_\ore$. Let $\bn_\orde\in\tPcal^\orde$, we have
\begin{align*}
\Qcal_\ore^{(\lambda)}[\bn](x)=\frac{\ind(x\leq c_e)\lambda^{x}\sum_{|\by|\leq b_v-x_e}\lambda^{|\by|}\bn_\orde(\by)}{\sum_{t\leq c_e}\lambda^t\sum_{|\by|\leq b_v-t}\lambda^{|\by|}\bn_\orde(\by)}
\end{align*}
Suppose that $|\bbeta_\orde|\geq b_v-c_e$, then the denominator of the expression above is equivalent as $\lambda\to\infty$ to $\lambda^{b_v}\sum_{b_v-c_e\leq|\by|\leq b_v}\bn_\orde(\by)>0$ and the numerator to $\ind(x\leq c_e)\lambda^{b_v}\sum_{|\by|=b_v-x_e}\bn_\orde(\by)+\operatorname o(\lambda^{b_v})$. Hence, $$\lim_{\lambda\to\infty}\Qcal_\ore^{(\lambda)}[\bn](x)=\frac{\ind(x\leq c_e)\sum_{|\by|=b_v-x_e}\bn_\orde(\by)}{\sum_{b_v-c_e\leq|\by|\leq b_v}\bn_\orde(\by)}.$$

Suppose now that $|\bbeta_\orde|<b_v-c_e$. The denominator is equivalent to $$\lambda^{c_e+|\bbeta_\orde|}\sum_{|\by|=|\bbeta_\orde|}\bn_\orde(\by)>0$$ and the numerator to $$\ind(x\leq c_e)\lambda^{x+|\bbeta_\orde|}\sum_{|\by|=|\bbeta_\orde|}\bn_\orde(\by)+\operatorname o(\lambda^{c_e+|\bbeta_\orde|}).$$ Thus, $$\lim_{\lambda\to\infty}\Qcal_\ore^{(\lambda)}[\bn](x)=\delta_{c_e}(x).$$
\end{proof}

\subsection{Finite graphs}
\begin{lemma}\label{lemma: adjacent sequences}
$\left(\bm^{2k}\right)_{k\in\N}$ and $\left(\bm^{2k+1}\right)_{k\in\N}$ are adjacent sequences in $\tPcallc^\orE$ for the $\lru$-order. We can then define $\bm^-=\lim\uparrow_{k\to\infty}\bm^{2k}\leq_\lru\bm^+=\lim\downarrow_{k\to\infty}\bm^{2k+1}$ in $\tPcallc^\orE$. They satisfy
\begin{align*}
\bm^+=\Rcal_G^{(\lambda)}[\bm^-]\text{ and }\bm^-=\Rcal_G^{(\lambda)}[\bm^+].
\end{align*}
\end{lemma}
\begin{proof}
First-of-all, it is clear that $\bm^k\in\tPcallc$ for all $k\in\N$, and for all $\ore$, $\bm_\ore^k\geq_\lru\delta_0$. In particular, $\bm^1\geq_\lru\bm^0$, and thus $\bm^2\leq_\lru\bm^1$ as $\Rcal_G^{(\lambda)}$ is non-increasing; iterating $\Rcal_G^{(\lambda)}$ we get also $\bm^3\geq_\lru\bm^2$. Thus,
\begin{align*}
\bm^0\leq_\lru\bm^2\leq_\lru\bm^3\leq_\lru\bm^1.
\end{align*}

Now, it is easy to see that applying $\Rcal_G^{(\lambda)}\circ\Rcal_G^{(\lambda)}$ to the equation above gives $\bm^2\leq_\lru\bm^4\leq_\lru\bm^5\leq_\lru\bm^3$; keeping iterating $\Rcal_G^{(\lambda)}\circ\Rcal_G^{(\lambda)}$ thus yields two adjacent sequences as claimed. Taking the limit $k\to\infty$ in $\bm^{2k}\leq_\lru\bm^{2k+1}$ gives $\bm^-\leq_\lru\bm^+$, and similarly letting $k\to\infty$ in $\bm^{2k+1}=\Rcal_G^{(\lambda)}[\bm^{2k}]$ and $\bm^{2k+2}=\Rcal_G^{(\lambda)}[\bm^{2k+1}]$ yields $\bm^+=\Rcal_G^{(\lambda)}[\bm^-]$ and $\bm^-=\Rcal_G^{(\lambda)}[\bm^+]$ respectively, due to Proposition~\ref{prop: continuity for log-concave inputs and limiting operator}.
\end{proof}

\begin{proposition*}[Unique fixed point in finite graphs at finite $\lambda$; Proposition~\ref{prop: unique fixed-point in finite graphs}]
In a finite graph $G$, the fixed-point equation $\bm=\Rcal_G^{(\lambda)}[\bm]$ admits a unique solution $\bm^{(\lambda)}\in\tPcallc^\orE$.
\end{proposition*}
\begin{proof}
As $\bm^-\leq_\lru\bm^+$ and $\Dcal_v$ is non-decreasing for all $v\in V$, it follows $\Dcal_v[\bm^-]\leq\Dcal_v[\bm^+]$ for all $v\in V$. Then, summing over all vertices of $G$ we get
\begin{align*}
\sum_{v\in V}\Dcal_v[\bm^-]&=\sum_{v\in V}\sum_{u\sim v}\frac{\sum_{x\in\N}xm_{\overrightarrow{uv}}^-(x)\Rcal_{\overrightarrow{vu}}[\bm^-](x)}{\sum_{x\in\N}m_{\overrightarrow{uv}}^-(x)\Rcal_{\overrightarrow{vu}}[\bm^-](x)}\\
&=\sum_{v\in V}\sum_{u\sim v}\frac{\sum_{x\in\N}x\Rcal_{\overrightarrow{uv}}[\bm^+](x)m_{\overrightarrow{vu}}^+(x)}{\sum_{x\in\N}\Rcal_{\overrightarrow{uv}}[\bm^+](x)m_{\overrightarrow{vu}}^+(x)}\\
&=\sum_{u\in V}\sum_{v\sim u}\frac{\sum_{x\in\N}x\Rcal_{\overrightarrow{uv}}[\bm^+](x)m_{\overrightarrow{vu}}^+(x)}{\sum_{x\in\N}\Rcal_{\overrightarrow{uv}}[\bm^+](x)m_{\overrightarrow{vu}}^+(x)}
=\sum_{u\in V}\Dcal_u[\bm^+]
\end{align*}

Hence, in fact, $\Dcal_v[\bm^-]=\Dcal_v[\bm^+]$ for all $v\in V$. As $\Dcal_v$ is strictly increasing for these inputs, it follows $\bm^-=\bm^+=\bm^{(\lambda)}$.

Finally, if $\bm\in\Pcal^\orE$ is another solution of $\bm=\Rcal_G^{(\lambda)}[\bm]$. Then $\bm\geq_\lru\bm^0$, and thus by iterating $\Rcal_G^{(\lambda)}$ we obtain $\bm^-\leq_\lru\bm\leq_\lru\bm^+$. Thus, $\bm=\bm^{(\lambda)}$.
\end{proof}

We do not include a proof of the following proposition, as it is not part of the development towards the main theorems, and is only mentionned here as a related result having interesting algorithmic consequences.
\begin{proposition*}[Correctness for finite bipartite graphs; Proposition~\ref{prop: Chertkov}]
In finite bipartite graphs,
\begin{align*}
\frac{1}{2}\lim\uparrow_{\lambda\to\infty}\sum_{v\in V}\Dcal_v[\bm^{(\lambda)}]=M(G).
\end{align*}
\end{proposition*}

\begin{proposition*}[Limit of $\lambda\to\infty$ (Proposition~\ref{prop: limit of $0$ temperature})]
$\bm^{(\lambda)}$ is non-decreasing in $\lambda$, and $\bm^{(\infty)}=\lim\uparrow_{\lambda\to\infty}\bm^{(\lambda)}\in\Pcallc^\orE$ is the minimal solution (for the $\lru$-order) of $\bm^{(\infty)}=\Qcal_G\circ\Rcal_G[\bm^{(\infty)}]$.
\end{proposition*}
\begin{proof}
We first show that $\lambda\mapsto\bm^{(\lambda)}$ and $\lambda\mapsto\lambda^{-\N}\centerdot\bm^{(\lambda)}$ are respectively non-decreasing and non-increasing, where $\lambda^{-\N}(x)=\lambda^{-x}$ for $x\in\N$. We proceed by induction: $m_\ore^0=\delta_0$ for all $\ore\in\orE$, hence $\lambda\mapsto\bm^0$ and $\lambda\mapsto\lambda^{-\N}\centerdot\bm^0$ are constant functions. Suppose now that $\lambda\mapsto\bm^k$ and $\lambda\mapsto\lambda^{-\N}\centerdot\bm^k$ are respectively non-decreasing and non-increasing, for some $k\in\N$. $\lambda^{-\N}\centerdot\bm^{k+1}=\Rcal_G[\bm^k]$. As $\lambda\mapsto\bm^k$ is non-increasing and $\Rcal_G$ is non-increasing, it follows $\lambda\mapsto\lambda^{-\N}\centerdot\bm^{k+1}$ is non-increasing. Similarly, let $\lambda'\geq\lambda$ and call $\bm'$ the messages associated with $\lambda'$. We have $\bm^{k+1}=\Qcal_G^{(\lambda)}[\lambda^{-\N}\centerdot\bm^k]\leq_\lru\Qcal_G^{(\lambda)}[{\lambda'}^{-\N}\centerdot\bm'^k]$ because $\bm^k$ is non-increasing in $\lambda$ and $\Qcal_G^{(\lambda)}$ is non-decreasing. Also, $\Qcal^{(\lambda)}$ is non-decreasing in $\lambda$, so $\Qcal_G^{(\lambda)}[{\lambda'}^{-\N}\centerdot\bm'^k]\leq_\lru\Qcal_G^{(\lambda')}[{\lambda'}^{-\N}\centerdot\bm'^k]=\bm'^{k+1}$.

Taking the limit $k\to\infty$, we obtain that $\lambda\mapsto\bm^{(\lambda)}$ and $\lambda\mapsto\lambda^{-\N}\centerdot\bm^{(\lambda)}$ are respectively non-decreasing and non-increasing. This allows us to define $\bm^{(\infty)}=\lim\uparrow_{\lambda\to\infty}\bm^{(\lambda)}\in\Pcallc^\orE$ and $\bn^{(\infty)}=\lim\downarrow_{\lambda\to\infty}\lambda^{-\N}\centerdot\bm^{(\lambda)}\in\tPcallc^\orE$. Passing to the limit $\lambda\to\infty$ in $\lambda^{-\N}\centerdot\bm^{(\lambda)}=\Rcal_G[\bm^{(\lambda)}]$ and in $\bm^{(\lambda)}=\Qcal_G^{(\lambda)}[\lambda^{-\N}\centerdot\bm^{(\lambda)}]$, we obtain
\begin{eqnarray}\label{eqn: two-step fixed-point}
\bn^{(\infty)}=\Rcal_G[\bm^{(\infty)}]&\text{ and }&\bm^{(\infty)}=\Qcal_G[\bn^{(\infty)}].
\end{eqnarray}

Let $\bm\in\Pcal^\orE$ be another solution of the two-step equation given by~(\ref{eqn: two-step fixed-point}). We have $m_\ore\geq_\lru\delta_0$, hence $\bm^{2k}\leq_\lru\left(\Qcal_G^{(\lambda)}\circ\Rcal_G\right)^k[\bm]\leq_\lru\bm$, where the first inequality is obtained by applying the non-decreasing operator $\Qcal_G^{(\lambda)}\circ\Rcal_G$ $k$ times and the second one follows from the fact $\Qcal_G^{(\lambda)}\circ\Rcal_G[\bm]$ is non-decreasing in $\lambda$, thus $\left(\Qcal_G^{(\lambda)}\circ\Rcal\right)^k[\bm]\leq_\lru\left(\Qcal_G\circ\Rcal\right)^k[\bm]=\bm$. Taking the limit $k\to\infty$ yields $\bm^{(\lambda)}\leq_\lru\bm$, and then $\lambda\to\infty$ gives $\bm^{(\infty)}\leq_\lru\bm$.
\end{proof}

The case of infinite unimodular graphs being more general, the proof of the following result is included only in this context (see Proposition~\ref{prop: BP estimate in unimodular random graphs}).
\begin{proposition*}[BP estimate in finite graphs; Proposition~\ref{prop: BP estimate in finite graphs}]
In a finite graph $G$, we have
\begin{align*}
\lim\uparrow_{\lambda\to\infty}\sum_{v\in V}\Dcal_v[\bm^{(\lambda)}]=\sum_{v\in V}\Dcal_v[\bm^{(\infty)}]=\sum_{v\in V}F_v(\balpha^{(\infty)})=\inf_{\balpha=\Scal_G\circ\Scal_G(\balpha)} \sum_{v\in V}F_v(\balpha),
\end{align*}
where $F_v(\balpha)=\min(b_v,|\balpha_\ordv|)+(b_v-|\balpha_\oldv|)^+$.
\end{proposition*}

\subsection{Infinite unimodular graphs}
\begin{proposition*}[Proposition~\ref{prop: unique fixed-point in unimodular graphs}]
Let $\rho\in\Ucal$ with $\overline b(\rho)<\infty$. Then, the fixed point equation $\bm=\Rcal^{(\lambda)}[\bm]$ admits a unique solution $\balpha^{(\lambda)}$ for any $\lambda\in\R^+$ for $\rho$-almost every marked graph $G$.
\end{proposition*}
\begin{proof}
The proof differs from that in the finite graph case in that we cannot sum $\Dcal_v$ over all the vertices of $v$ anymore. Instead, we use the MTP for $f(G,r,v)=\frac{\sum_{x\in\N}xm_{\overrightarrow{vr}}^-(x)\Rcal_{\overrightarrow{rv}}[\bm^-](x)}{\sum_{x\in\N}m_{\overrightarrow{vr}}^-(x)\Rcal_{\overrightarrow{rv}}[\bm^-](x)}$:
\begin{align*}
\int\Dcal_r[\bm^-]d\rho([G,r])&=\int\sum_{v\sim r}\frac{\sum_{x\in\N}xm_{\overrightarrow{vr}}^-(x)\Rcal_{\overrightarrow{rv}}[\bm^-](x)}{\sum_{x\in\N}m_{\overrightarrow{vr}}^-(x)\Rcal_{\overrightarrow{rv}}[\bm^-](x)}d\rho([G,r])\\
&=\int\sum_{v\sim r}\frac{\sum_{x\in\N}x\Rcal_{\overrightarrow{vr}}[\bm^+](x)m_{\overrightarrow{rv}}^+(x)}{\sum_{x\in\N}\Rcal_{\overrightarrow{vr}}[\bm^-](x)m_{\overrightarrow{rv}}^-(x)}d\rho([G,r])\\
&=\int\sum_{v\sim r}\frac{\sum_{x\in\N}x\Rcal_{\overrightarrow{rv}}[\bm^+](x)m_{\overrightarrow{vr}}^+(x)}{\sum_{x\in\N}\Rcal_{\overrightarrow{rv}}[\bm^-](x)m_{\overrightarrow{vr}}^-(x)}d\rho([G,r])\\
&=\int\Dcal_r[\bm^+]d\rho([G,r])
\end{align*}

Because $\overline b(\rho)<\infty$, these expectations are finite, and as $\Dcal_v$ is strictly increasing for all $v\in V$ it yields that $\bm_{\overrightarrow{\d r}}^-=\bm_{\overrightarrow{\d r}}^+$, $\rho$-almost surely. By Lemma~2.3 \cite{Aldous07}, this result extends to the edges incoming to $\rho$-almost every vertex in $G$, hence $\bm^-=\bm^+=\bm^{(\lambda)}$ $\rho$-a.s.
\end{proof}

\begin{lemma}\label{lemma: computation trick}
Let $\bm_\ordv\in\Pcal^\ordv$. We have
\begin{align*}
\ind(|\bbeta_\ordv|>b_v)\sum_{\ore\in\ordv}\frac{\sum_x (x-\Scal_\ole(\bbeta))\Qcal_\ole[\bn](x) n_\ore(x)}{\sum_x \Qcal_\ole[\bn](x) n_\ore(x)}=\left(b_v-|\bScal_\oldv(\bbeta)|\right)^+
\end{align*}
\end{lemma}

\begin{proof}
Both sides are equal to $0$ unless $|\bbeta_\ordv|>b_v$. We have
\begin{align*}
&\ind(|\bbeta_\ordv|>b_v)\sum_{\ore\in\ordv}\frac{\sum_x (x-\Scal_\ole(\bbeta))\Qcal_\ole[\bn](x) n_\ore(x)}{\sum_x \Qcal_\ole[\bn](x) n_\ore(x)}\\
&=\ind(|\bbeta_\ordv|>b_v)\frac{\sum_{|\bx|=b_v}(|\bx|-|\bScal_\oldv(\bbeta)|)\bn_\ordv(\bx)}{\sum_{|\bx|=b_v}\bn_\ordv(\bx)}\\
&=(b_v-|\bScal_\oldv(\bbeta)|)\ind(|\bbeta_\ordv|>b_v)\\
&=(b_v-|\bScal_\oldv(\bbeta)|)^+
\end{align*}
where the last line follows from the fact $\Scal_\ole(\bbeta)=[b_v-|\bbeta_\ordv|+\beta_\ore]_0^{c_e}\leq\beta_\ore$ if $|\bbeta_\ordv|>b_v$, the inequality being strict whenever $\beta_\ore>0$, and it is then easy to check that $\sum_{\ole\in\oldv}\Scal_\ole(\bbeta)\leq\sum_{\ore\in\ordv}\beta_\ore-(|\bbeta_\ordv|-b_v)\leq b_v$.
\end{proof}

\begin{lemma}\label{lemma: computing degree}
Let $F_v(\balpha)=\min(b_v,|\balpha_\ordv|)+(b_v-|\balpha_\oldv|)^+$, $\forall v\in V$. Consider $\bm'=\Qcal_G\circ\Rcal_G[\bm]$ such that $\balpha(\bm)=\balpha(\bm')$.
\begin{itemize}
\item If $\bm'\geq_\lru\bm$, then $\int\Dcal_r[\bm]d\rho([G,r])\leq\int F_r(\balpha(\bm))d\rho([G,r])$.
\item If $\bm'\leq_\lru\bm$, then $\int\Dcal_r[\bm]d\rho([G,r])\geq\int F_r(\balpha(\bm))d\rho([G,r])$.
\end{itemize}
\end{lemma}

\begin{proof}
Let $\bn=\Rcal_G[\bm]$. For any $v\in V$, $|\balpha_\ordv(\bm)|\leq b_v$ is equivalent to $|\bbeta_\oldv(\bn)|\geq b_v$. We have
\begin{align*}
\int\Dcal_r[\bm]d\rho([G,r])&=\int\Bigg(\min(b_r,|\balpha_{\overrightarrow{\d r}}(\bm)|)\\
&+\ind(|\balpha_{\overrightarrow{\d r}}(\bm)|<b_r)\sum_{v\in\d r}\frac{\sum_{x\in\N}(x-\alpha(m_{\overrightarrow{vr}}))m_{\overrightarrow{vr}}(x)\Rcal_{\overrightarrow{rv}}[\bm](x)}{\sum_{x\in\N}m_{\overrightarrow{vr}}(x)\Rcal_{\overrightarrow{rv}}[\bm](x)}\Bigg)d\rho([G,r])
\end{align*}

Furthermore, for any $v\sim r$, $\sum_{x\in\N}(x-\alpha(m_{\overrightarrow{vr}}))m_{\overrightarrow{vr}}(x)\Rcal_{\overrightarrow{rv}}[\bm](x)$ is non-zero only if $\alpha(m_{\overrightarrow{vr}})<\Scal_{\overrightarrow{rv}}(\balpha(\bn))$, which also implies $|\balpha_{\overrightarrow{\d r}}(\bm)|<b_r$. Thus, the second term in the expression of $\int\Dcal_r[\bm]d\rho([G,r])$ above is equal to
\begin{align*}
\int\sum_{v\sim r}\frac{\sum_{x\in\N}(x-\alpha(m_{\overrightarrow{vr}}))m'_{\overrightarrow{vr}}(x)\Rcal_{\overrightarrow{rv}}[\bm](x)}{\sum_{x\in\N}m'_{\overrightarrow{vr}}(x)\Rcal_{\overrightarrow{rv}}[\bm](x)}\ind(\alpha_{\overrightarrow{vr}}<\Scal_{\overrightarrow{rv}}(\balpha(\bn)))d\rho([G,r])
\end{align*}

Suppose first that $\bm'\geq_\lru\bm$. Then, for any $v\sim r$, we have $m'_{\overrightarrow{vr}}\centerdot\Rcal_{\overrightarrow{rv}}[\bm]\geq_\lru m_{\overrightarrow{vr}}\centerdot\Rcal_{\overrightarrow{rv}}[\bm]$, as also $\alpha(m'_{\overrightarrow{vr}})=\alpha(m_{\overrightarrow{vr}})$. Hence, we have
\begin{align*}
\frac{\sum_{x\in\N}(x-\alpha(m_{\overrightarrow{vr}}))m'_{\overrightarrow{vr}}(x)\Rcal_{\overrightarrow{rv}}[\bm](x)}{\sum_{x\in\N}m'_{\overrightarrow{vr}}(x)\Rcal_{\overrightarrow{rv}}[\bm](x)}\geq\frac{\sum_{x\in\N}(x-\alpha(m_{\overrightarrow{vr}}))m_{\overrightarrow{vr}}(x)\Rcal_{\overrightarrow{rv}}[\bm](x)}{\sum_{x\in\N}m_{\overrightarrow{vr}}(x)\Rcal_{\overrightarrow{rv}}[\bm](x)}
\end{align*}

It follows
\begin{align*}
&\int\sum_{v\sim r}\frac{\sum_{x\in\N}(x-\alpha(m_{\overrightarrow{vr}}))m'_{\overrightarrow{vr}}(x)\Rcal_{\overrightarrow{rv}}[\bm](x)}{\sum_{x\in\N}m'_{\overrightarrow{vr}}(x)\Rcal_{\overrightarrow{rv}}[\bm](x)}\ind(\alpha(m_{\overrightarrow{vr}})<\Scal_{\overrightarrow{rv}}(\balpha(\bn)))d\rho([G,r])\\
&\geq\int\sum_{v\sim r}\frac{\sum_{x\in\N}(x-\alpha(m_{\overrightarrow{vr}}))m'_{\overrightarrow{vr}}(x)\Rcal_{\overrightarrow{rv}}[\bm](x)}{\sum_{x\in\N}m'_{\overrightarrow{vr}}(x)\Rcal_{\overrightarrow{rv}}[\bm](x)}\ind(\alpha(m'_{\overrightarrow{vr}})<\Scal_{\overrightarrow{rv}}(\balpha(\bn)))d\rho([G,r])\\
&=\int\sum_{v\sim r}\frac{\sum_{x\in\N}(x-\Scal_{\overrightarrow{vr}}(\bbeta(\bn)))\Qcal_{\overrightarrow{vr}}[\bn](x)n_{\overrightarrow{rv}}(x)}{\sum_{x\in\N}\Qcal_{\overrightarrow{vr}}[\bn](x)n_{\overrightarrow{rv}}(x)}\ind(\Scal_{\overrightarrow{vr}}(\bbeta(\bn))<\beta(n_{\overrightarrow{rv}}))d\rho([G,r])\\
&=\int\sum_{v\sim r}\frac{\sum_{x\in\N}(x-\Scal_{\overrightarrow{rv}}(\bbeta(\bn)))\Qcal_{\overrightarrow{rv}}[\bn](x)n_{\overrightarrow{vr}}(x)}{\sum_{x\in\N}\Qcal_{\overrightarrow{rv}}[\bn](x)n_{\overrightarrow{vr}}(x)}\ind(\Scal_{\overrightarrow{rv}}(\bbeta(\bn))<\beta(n_{\overrightarrow{vr}}))d\rho([G,r]),
\end{align*}
where the last equality follows from by the Mass-Transport Principle. Given that $\Scal_{\overrightarrow{rv}}(\bbeta(\bn))<\beta(n_{\overrightarrow{vr}})$ implies $|\bbeta_{\overrightarrow{\d r}}(\bn)|>b_r$, $\Qcal_{\overrightarrow{rv}}[\bn](x)n_{\overrightarrow{vr}}(x)$ is non-zero only if $\Scal_{\overrightarrow{rv}}(\bbeta(\bn))\leq\beta(n_{\overrightarrow{vr}})$, and $|\bbeta_{\overrightarrow{\d r}}(\bn)|>b_r$ implies $\Scal_{\overrightarrow{rv}}(\bbeta(\bn))\leq\beta(n_{\overrightarrow{vr}})$, we have in fact
\begin{align*}
&\sum_{v\sim r}\frac{\sum_{x\in\N}(x-\Scal_{\overrightarrow{rv}}(\bbeta(\bn)))\Qcal_{\overrightarrow{rv}}[\bn](x)n_{\overrightarrow{vr}}(x)}{\sum_{x\in\N}\Qcal_{\overrightarrow{rv}}[\bn](x)n_{\overrightarrow{vr}}(x)}\ind(\Scal_{\overrightarrow{rv}}(\bbeta(\bn))<\beta(n_{\overrightarrow{vr}}))\\
&=\ind(|\bbeta_{\overrightarrow{\d r}}(\bn)|>b_r|)\sum_{v\sim r}\frac{\sum_{x\in\N}(x-\Scal_{\overrightarrow{rv}}(\bbeta(\bn)))\Qcal_{\overrightarrow{rv}}[\bn](x)n_{\overrightarrow{vr}}(x)}{\sum_{x\in\N}\Qcal_{\overrightarrow{rv}}[\bn](x)n_{\overrightarrow{vr}}(x)}\\
&=\ind(|\bbeta_{\overrightarrow{\d r}}(\bn)|>b_r|)(b_r-|\bScal_{\overleftarrow{\d r}}(\bbeta(\bn))|)^+=(b_r-|\balpha_{\overleftarrow{\d r}}(\bm')|)^+,
\end{align*}
according to Lemma~\ref{lemma: computation trick}. We thus obtained that
\begin{align*}
\int\Dcal_r[\bm]d\rho([G,r])\leq\int F_r(\balpha(\bm))d\rho([G,r])\text{ if }\bm'\geq_\lru\bm.
\end{align*}

The proof for the case $\bm'\leq_\lru\bm$ is identical.
\end{proof}

\begin{proposition*}[BP estimate in unimodular random graphs; Proposition~\ref{prop: BP estimate in unimodular random graphs}]
Let $\rho\in\Ucal$ with $\overline b(\rho)<\infty$, we have
\begin{align*}
\lim\uparrow_{\lambda\to\infty}\int\Dcal_r[\bm^{(\lambda)}]d\rho([G,r])&=\int\Dcal_r[\bm^{(\infty)}]d\rho([G,r])=\int F_r(\balpha^{(\infty)})d\rho([G,r])\\
&=\inf_{\balpha=\Scal_G\circ\Scal_G(\balpha)}\int F_r(\balpha)d\rho([G,r]),
\end{align*}
where $F_v(\balpha)=\min(b_v,|\balpha_\ordv|)+(b_v-|\balpha_\oldv|)^+$.
\end{proposition*}
\begin{proof}
For any $v\in V$, as $\Dcal_v$ is continuous for inputs in $\tPcallc$, such as $\bm^{(\lambda)}$, hence we have $\lim\uparrow_{\lambda\to\infty}\Dcal_v[\bm^{(\lambda)}]=\Dcal_v[\bm^{(\infty)}]$ and the first equality follows by monotone convergence. Lemma~\ref{lemma: computing degree} shows
\begin{align*}
\int\Dcal_r[\bm^{(\infty)}]d\rho([G,r])=\int F_r(\balpha(\bm^{(\infty)}))d\rho([G,r]).
\end{align*}

It is clear that $\balpha(\bm^{(\infty)})$ satisfies $\balpha(\bm^{(\infty)})=\Scal_G\circ\Scal_G(\balpha(\bm^{(\infty)}))$, as $\bm^{(\infty)}$ satisfies equation~(\ref{eqn: two-step fixed-point}). Consider now $\balpha\in\N^\orE$ such that $\balpha=\Scal_G\circ\Scal_G(\balpha)$ and the measure $\rho$ with marks $\balpha$ is unimodular. We define $\bm^{(0)}_\ore=\delta_{\alpha_\ore}$ and $\bm^{(k+1)}=\Qcal_G\circ\Rcal_G[\bm^{(k)}]$. We have $\balpha(\bm^{(k)})=\balpha$ for any $k\in\N$, and then necessarily $\bm^{(k)}\geq_\lru\bm^{(0)}$ (it suffices to look at the support of $m_\ore^{(k)}$ and $m_\ore^{(0)}$ to check this). It follows, because $\Qcal_G\circ\Rcal_G$ is non-decreasing, that $\left(\bm^{(k)}\right)_{k\in\N}$ is a non-decreasing sequence, and thus $\int\Dcal_r[\bm^{(k)}]d\rho([G,r])$ is non-decreasing in $k$ (as $\Dcal_r$ is non-decreasing).

We define $\bm=\lim\uparrow_{k\in\N}\bm^{(k)}$. Clearly, $\balpha(\bm)\geq\balpha$ and $\bm=\Qcal_G\circ\Rcal_G(\bm)$. Moreover, by monotone convergence, we have
\begin{align*}
\lim\uparrow_{k\to\infty}\int\Dcal_r[\bm^{(k)}]d\rho([G,r])&=\int\Dcal_r[\bm]d\rho([G,r])=\int F_r(\balpha(\bm))d\rho([G,r])\\
&\geq\int\Dcal_r[\bm^{(\infty)}]d\rho([G,r])=\int F_r(\balpha(\bm^{(\infty)}))d\rho([G,r]),
\end{align*}
where we used Lemma~\ref{lemma: computing degree} for $\bm$ and for $\bm^{(\infty)}$ and Proposition~\ref{prop: limit of $0$ temperature} together with the fact $\Dcal_r$ is non-decreasing to get the inequality. For any $k\in\N$ we have $\bm^{(k+1)}=\Qcal_G\circ\Rcal_G[\bm^{(k)}]\geq_\lru\bm^{(k)}$, thus applying Lemma~\ref{lemma: computing degree} we obtain that
\begin{align*}
\int F_r(\balpha)d\rho([G,r])&=\int F_r(\balpha(\bm^{(k)}))d\rho([G,r])\geq\int\Dcal_r[\bm^{(k)}]d\rho([G,r])\\
&\nearrow_{k\to\infty}\int\Dcal_r[\bm^{(\infty)}]d\rho([G,r])\geq\int F_r(\balpha(\bm^{(\infty)}))d\rho([G,r]),
\end{align*}
which completes the proof.
\end{proof}

\subsection{From finite graphs to infinite unimodular trees}
The following proposition is not proved here. The interested reader can refer to \cite{bls11}, \cite{Salez11} and \cite{Lelarge12}, where similar results appear.
\begin{proposition*}[Asymptotic correctness for large, sparse random graphs; Proposition~\ref{prop: asymptotic correctness for unimodular random trees}]
Let $G_n=(V_n,E_n)_n$ be a sequence of finite marked graphs with random weak limit $\rho$ concentrated on unimodular trees, with $\overline b(\rho)<\infty$. Then,
\begin{align*}
\lim_{n\to\infty}\frac{2M_n}{|V_n|}=\int\Dcal_r[\bm^{(\infty)}]d\rho([G,r])=\inf_{\balpha=\Scal_G\circ\Scal_G(\balpha)}\int F_r(\balpha)d\rho([G,r]).
\end{align*}
\end{proposition*}

\subsection{Galton-Watson trees}
The main theorem follows quite straightforwardly from Propositions~\ref{prop: BP estimate in unimodular random graphs} and \ref{prop: asymptotic correctness for unimodular random trees}. The missing steps are standard and can be found in \cite{Lelarge12}; they resemble much the computation done in the proof of Proposition~\ref{prop: BP estimate in unimodular random graphs}.
\begin{theorem*}[Maximum allocation for bipartite Galton-Watson limits; \ref{th: maximum allocation for bipartite Galton-Watson limits}]
Provided $\EE[W^A]$ and $\EE[W^B]$ are finite, the limit $\Mcal(\Phi^A,\Phi^B):=\lim_{n\to\infty}M(G_n)/|A_n|$ exists and equals
$$
\begin{array}{ll}
\Mcal(\Phi^A,\Phi^B)=&\inf
\left\{\EE\left[\min\left\{W^A,\sum_{i=1}^{D^A}X_i(C_i^A)\right\}\right]\right.\\
&+\left. \frac{\EE[D^A]}{\EE[D^B]}\EE\left[\left(W^B-\sum_{i=1}^{D^B}\left[W^B-\sum_{j\neq i}Y_j(C_j^B)\right]_0^{C_i^B}\right)^+\ind\left({W^B<\sum_{i=1}^{D^B}C_i^B}\right]\right)\right\}
\end{array}
$$
where for all $i$, $\left(X_i(c),Y_i(c)\right)_{c\in\N}$ is an independent copy of $\left(X(c),Y(c)\right)_{c\in\N}$, and the infimum is taken over distributions for $\left(X(c),Y(c)\right)_{c\in\N}$ satisfying the RDE
\begin{align*}
Y(c)=\left\{\left[\widetilde W^A-\sum_{i=1}^{\widetilde D^A}X_i(\widetilde C_i^A)\right]_0^c\Bigg|C_0^A=c\right\}; X(c)=\left\{\left[\widetilde W^B-\sum_{i=1}^{\widetilde D^B}Y_i(\widetilde C_i^B)\right]_0^c\Bigg|C_0^B=c\right\}.
\end{align*}
\end{theorem*}
\begin{proof}
Propositions~\ref{prop: BP estimate in unimodular random graphs} and \ref{prop: asymptotic correctness for unimodular random trees} together give that
\begin{align*}
\lim_{n\to\infty}\frac{2M(G_n)}{|A_n|+|B_n|}=\inf_{\balpha=\Scal_G\circ\Scal_G(\balpha)}\int F_r(\balpha)d\rho([G,r])=\int\Dcal_r[\bm^{(\infty)}]d\rho([G,r])
\end{align*}

We introduce the probability measures $\rho^A$ and $\rho^B$ on $\Ucal$ by conditioning on the root being in $A$ or $B$: $\rho^A([G,r])=\rho([G,r])\ind(r\in A)\frac{\EE[D^A]+\EE[D^B]}{\EE[D^B]}$, and similarly for $\rho^B$.

For $\lambda\in\R^+$, applying the MTP to $\rho\in\Ucal$ with $f^A(G,r,v)=\frac{\sum_{x\in\N}xm_{\overrightarrow{vr}}^{(\lambda)}(x)\Rcal_{\overrightarrow{rv}}[\bm^{(\lambda)}](x)}{\sum_{x\in\N}m_{\overrightarrow{vr}}^{(\lambda)}(x)\Rcal_{\overrightarrow{rv}}[\bm^{(\lambda)}](x)}\ind(r\in A)$, we obtain
\begin{align*}
\int\Dcal_r[\bm^{(\lambda)}]d\rho^A([G,r])&=\frac{\EE[D^A]+\EE[D^B]}{\EE[D^B]}\int\sum_v f^A(G,r,v)d\rho([G,r])\\
&=\frac{\EE[D^A]+\EE[D^B]}{\EE[D^B]}\int\sum_v f^A(G,v,r)d\rho([G,r])\\
&=\frac{\EE[D^A]+\EE[D^B]}{\EE[D^B]}\int\sum_v f^B(G,r,v)d\rho([G,r])\\
&=\frac{\EE[D^A]}{\EE[D^B]}\int\Dcal_r[\bm^{(\lambda)}]d\rho^B([G,r])
\end{align*}

Letting $\lambda\to\infty$ yields $\int\Dcal_r[\bm^{(\infty)}]d\rho^A([G,r])=\int\Dcal_r[\bm^{(\infty)}]d\rho^B([G,r])$, which shows
\begin{align*}
\lim_{n\to\infty}\frac{M(G_n)}{|A_n|}=\int\Dcal_r[\bm^{(\infty)}]d\rho^A([G,r]).
\end{align*}

We then follow exactly the steps in the proof of Proposition~\ref{prop: BP estimate in unimodular random graphs} for $\rho^A$ instead of $\rho$. This gives
\begin{align*}
\int\Dcal_r[\bm^{(\infty)}]d\rho^A([G,r])=\inf_{\balpha=\Scal_G\circ\Scal_G(\balpha)}&\left\{\int\min(b_r,|\balpha_{\overrightarrow{\d r}}|)d\rho^A([G,r])\right.\\
&+\left.\frac{\EE[D^A]}{\EE[D^B]}\int(b_r-|\balpha_{\overleftarrow{\d r}}|)d\rho^B([G,r])\right\}
\end{align*}

As $G$ is an unimodular tree, for any vertex $v\in V$, all the components of $\balpha_\ordv$ can be chosen independently (as they are independent in $\balpha^{(\infty)}_\ordv$, which achieves the infimum). Then, for $\ore$ incoming to $v$, $\alpha_\ore$ is determined only from the subtree stemming from the tail of $\ore$; furthermore it satisfies $\alpha_\ore=\Scal_\ore\circ\Scal_\orde[\balpha]$. However, the distribution of the subtree at the tail of an $\ore'$ which is an input to $\Scal_\orde$ is the same as that of the subtree at the tail of $\ore$, by the two-step branching property of the bipartite Galton-Watson tree $G$. This implies that, for $\ore$ incoming to a root $r\in A$, $\alpha_\ore$ is solution of the two-step RDE given in the statement of the theorem. As detailed in Lemma~6 of \cite{Aldous05}, there is actually a one-to-one mapping between the solutions of $\balpha=\Scal_G\circ\Scal_G[\balpha]$ on a Galton-Watson tree $G$ and the solutions of the RDE considered here. This completes the proof.
\end{proof}

\subsection{Cuckoo hashing}
\begin{theorem*}[Threshold for $(k,l,r)$-orientability of $h$-uniform hypergraphs; Theorem~\ref{th: application to cuckoo hashing}]
Let $h,k,l,r$ be positive integers such that $k,l\geq r$, $(h-1)r\geq l$ and $k+(h-2)r-l>0$ (i.e. at least one of the inequalities among $k\geq r$ and $(h-1)r\geq l$ is strict). We define $\Phi^A$ and $\Phi^B_\tau$ by $(h,l,\{r\})\sim\Phi^A$ and $(\operatorname{Poi}(\tau h),k,\{r\})\sim\Phi^B_\tau$, and
\begin{align*}
\tau^*_{h,k,l,r}=\sup\left\{\tau:\Mcal(\Phi^A,\Phi^B_\tau)<l\right\}.
\end{align*}

Then,
\begin{align*}
\lim_{n\to\infty}\PP\left(H_{n,\lfloor\tau n\rfloor,h}\text{ is }(k,l,r)\text{-orientable}\right)=\left\{\begin{array}{ll}
 1\text{ if }\tau<\tau^*_{h,k,l,r}\\
 0\text{ if }\tau>\tau^*_{h,k,l,r}
\end{array}\right.
\end{align*}
\end{theorem*}
\begin{proof}
For any $h$-uniform hypergraph $H_n$ on $n$ vertices, we let $G_n=(A_n\cup B_n,E_n)$ be the associated bipartite graph, where $B_n$ contains the vertices of $H_n$ and $A_n$ the hyperedges of $H_n$. Let $|B_n|=n$, and $|A_n|=m=\lfloor\tau n\rfloor$ for some $\tau$. First-of-all, it is clear by coupling that $\tau\mapsto\Mcal(\Phi^A,\Phi_\tau^B)$ as defined in Theorem~\ref{th: maximum allocation for bipartite Galton-Watson limits}, is a non-decreasing function. Let then $\tau>\tau^*_{h,k,l,r}$. Then, by Theorem~\ref{th: maximum allocation for bipartite Galton-Watson limits}, we have
\begin{eqnarray*}
\lim_{n\to\infty}\frac{M(G_n)}{|A_n|}<l,
\end{eqnarray*}
which immediately implies that $G_n$ is a.a.s. not $(k,l,r)$-orientable.

Let now $\tau<\tau^*_{h,k,l,r}$. According to Theorem~\ref{th: maximum allocation for bipartite Galton-Watson limits} again, we have $\lim_{n\to\infty}\frac{M(G_n)}{|A_n|}=l$ but there may still exist $\operatorname o(n)$ hyperedges which are not $(l,r)$-oriented. We will then rely on specific properties of $H_{n,m,h}$ to show that a.a.s. all hyperedges are $(l,r)$-oriented. We follow here a similar path as in \cite{Gao10,Lelarge12}. It is easier to work with a different model of hypergraphs, that we call $H_{n,p,h}$, and that is essentially equivalent to the $H_{n,\lfloor\tau n\rfloor,h}$ model \cite{Kim08}: each possible $h$-hyperedge is included independently with probability $p$, with $p=\tau h/\binom{n-1}{h-1}$.

We let $\tilde\tau$ be such that $\tau<\tilde\tau<\tau^*_{h,k,l,r}$, and consider the bipartite graph $\tilde G_n=(\tilde A_n\cup B_n,\tilde E_n)$ obtained from $H_{n,\tilde p,h}$ with $\tilde p=\tilde\tau h/\binom{n-1}{h-1}$. Consider a maximum allocation $\tilde\bx\in\N^{\tilde E_n}$ of $\tilde G_n$. We say that a vertex of $w\in\tilde A_n$ (resp. a vertex $w\in B_n$) is covered if $\sum_{e\in\d w}\tilde x_e=l$ (resp. $\sum_{e\in\d w}\tilde x_e=k$); we also say that an edge $e\in\tilde E_n$ is saturated if $\tilde x_e=c$.

Let $v$ be a vertex in $\tilde A_n$ that is not covered. We define $K(v)$ as the minimum subgraph of $\tilde G_n$ such that:
\begin{itemize}
\item $v$ belongs to $K(v)$;
\item all the unsaturated edges adjacent to a vertex in $\tilde A_n\cap K(v)$ belong to $K(v)$ (and thus their endpoints in $B_n$ also belongs to $K(v)$);
\item all the edges $e$ for which $\tilde x_e>0$ and that are adjacent to a vertex in $B_n\cap K(v)$ belong to $K(v)$ (and so do their endpoints in $\tilde A_n$).
\end{itemize}
The subgraph $K(v)$ defined in this way is in fact constitued of $v$ and all the paths starting from $v$ and alternating between unsaturated edges and edges $e$ with $\tilde x_e>0$ (we call such a path an alternating path). It is then easy to see that all the vertices in $B_n\cap K(v)$ must be covered, otherwise we could obtain a strictly larger allocation by applying the following change: take the path $(e_1,\ldots,e_{2t+1})$ between $v$ and an unsaturated vertex in $B_n\cap K(v)$; add $1$ to each $\tilde x_{e_i}$ for $i$ odd, and remove $1$ from each $\tilde x_{e_i}$ for $i$ even; all these changes are possible due to the way the edges in $K(v)$ have been chosen, and the resulting allocation has size larger by $1$ than $|\tilde\bx|$.

We will now show that the subgraph $K(v)$ is dense, in the sense that the average induced degree of its vertices is strictly larger than $2$. We first show that all the vertices in $K(v)$ have degree at least $2$. We have $(h-1)r\geq l$ and $v$ is not covered, hence $v$ has at least two adjacent edges in $\tilde G_n$ which are not saturated, thus the degree of $v$ in $K(v)$, written $\operatorname{deg}_{K(v)}v$, is at least $2$. Let $w$ be a vertex in $B_n\cap K(v)$. By definition, there is an edge $e\in\d w\cap K(v)$ through which $w$ is reached from $v$ in an alternating path, and $\tilde x_e<r$. Then, because $\sum_{e\in\d w}\tilde x_e=k$ and $k\geq r$ there must be another edge $e'$ adjacent to $w$ such that $\tilde x_{e'}>0$; such an edge belongs to $K(v)$ and thus $w$ is at least of degree $2$ in $K(v)$. Let now $w$ be a vertex in $\tilde A_n\cap K(v)$, $w\neq v$. By definition, there must exist an edge $e\in\d w\cap K(v)$ such that $\tilde x_e>0$. Because $(h-1)r\geq l$ and $\tilde x_e>0$ there must be another edge $e'$ adjacent to $w$ such that $\tilde x_{e'}<r$; $e'$ belongs to $K(v)$ and thus $\operatorname{deg}_{K(v)}w\geq2$.

Consider a path $\left(e_1=(v_1v_2),\ldots,e_t=(v_tv_{t+1})\right)$ in $K(v)$ such that $v_1\in\tilde A_n\cap K(v)$ and any two consecutive edges in the path are distinct. We will show that at least one vertex out of $2r$ consecutive vertices along this path must have degree at least $3$ in $K(v)$, by showing that $\tilde x_{e_{2(i+1)+1}}<\tilde x_{e_{2i+1}}$ provided $v_{2(i+1)}$ and $v_{2(i+1)+1}$ have degree $2$ in $K(v)$ for all $i$. $v_{2(i+1)}\in B_n\cap K(v)$ must be covered, so if $\operatorname{deg}_{K(v)}v_{2(i+1)}=2$ we must have $\tilde x_{e_{2(i+1)}}=k-\tilde x_{e_{2i+1}}$. Then, if $\operatorname{deg}_{K(v)}v_{2(i+1)+1}=2$, all the edges adjacent to $v_{2(i+1)+1}$ except $e_{2(i+1)}$ and $e_{2(i+1)+1}$ must be saturated, thus we must also have $(h-2)r+\tilde x_{e_{2(i+1)}}+\tilde x_{e_{2(i+1)+1}}\leq l$. This immediately yield $\tilde x_{e_{2(i+1)+1}}+\left\{k+(h-2)r-l\right\}\leq\tilde x_{e_{2i+1}}$, and thus $\tilde x_{e_{2(i+1)+1}}<\tilde x_{e_{2i+1}}$ as claimed. But $\tilde x_{e_{2i+1}}<r$ and so $\tilde x_{e_{2i+2r+1}}\leq-1$ if the hypothesis that all the vertices encountered meanwhile have degree $2$ in $K(v)$ is correct, which is thus not possible. Note that we did not need to assume that the path considered was vertex-disjoint, hence it is not possible that $K(v)$ is reduced to a single cycle.

We will now count vertices and edges of $K(v)$ in a way that clearly shows that the number of edges in $K(v)$ is at least $\gamma$ times its number of vertices, with $\gamma>1$. We can always see $K(v)$ as a collection $P$ of edge-disjoint paths, with all vertices interior to a path of degree $2$ in $K(v)$ and the extremal vertices of a path having degree at least $3$ in $K(v)$. To form $K(v)$ we would simply need to merge the extremal vertices of some of these paths. We have shown before that each path in $P$ has at most $2r$ vertices. Let $p=\left(e_1=(v_1v_2),\ldots,e_t=(v_tv_{t+1})\right)$ be a path in $P$, we let $\theta_E(p)=t$ be the number of edges in $p$ and $\theta_V(p)=\sum_{e_i\in p}\frac{1}{\operatorname{deg}_{K(v)}v_i}+\frac{1}{\operatorname{deg}_{K(v)}v_{i+1}}$ be a partial count of the vertices in $p$ (all the interior vertices are counted as $1$ but the extremal vertices are only partially counted in $\theta_V(p)$, as they belong to many different paths). We have $\theta_V(p)=t-1+\frac{1}{\operatorname{deg}_{K(v)}v_1}+\frac{1}{\operatorname{deg}_{K(v)}v_{t+1}}\leq t-1+\frac{2}{3}$. Hence,
\begin{eqnarray*}
\frac{\theta_E(p)}{\theta_V(p)}\geq\frac{t}{t-1+\frac{2}{3}}\geq\frac{1}{1-\frac{1}{6r}}>1.
\end{eqnarray*}

Furthermore, it is easy to see that
\begin{eqnarray*}
\sum_{p\in P}\theta_E(p)&=&\text{ number of edges in }K(v),\\
\sum_{p\in P}\theta_V(p)&=&\text{ number of vertices in }K(v),
\end{eqnarray*}
which shows that the number of edges in $K(v)$ is at least $\gamma=\frac{1}{1-\frac{1}{6r}}>1$ times the number of vertices in $K(v)$.

Now, it is classical that any subgraph of a sparse random graph like $\tilde G_n$ with a number of edges equal to at least $\gamma>1$ times its number of vertices must contain at least a fraction $\epsilon>0$ of the vertices of $\tilde G_n$, with probability tending to $1$ as $n\to\infty$ (see \cite{Kim08,Gao10}). Therefore, $K(v)$ contains at least a fraction $\epsilon'>0$ of the vertices in $\tilde A_n$.

There exists a natural coupling between $H_{n,p,h}$ and $H_{n,\tilde p,h}$: we can obtain $H_{n,p,h}$ from $H_{n,\tilde p,h}$ by removing independently each hyperedge with probability $\tilde p-p>0$. This is equivalent to removing independently with probability $\tilde p-p$ each vertex in $\tilde A_n$. We let $\operatorname{gap}_n=l|\tilde A_n|-M(\tilde G_n)=\operatorname o(n)$. For any uncovered vertex $v$ in $\tilde A_n$ we can construct a subgraph $K(v)$ as above. If we remove a vertex $w$ in $\tilde A_n\cap K(v)$ for such a $v$, then either this vertex $w$ is itself uncovered, and then $\operatorname{gap}_n$ is decreased by at least $1$, or $w$ is covered and then it must belong to an alternating path starting from $v$ and we can construct a new allocation with size equal to that of $\tilde x$ and in which $w$ is uncovered and there is one more unit of weight on one of the edges adjacent to $v$, hence removing $w$ will also reduce $\operatorname{gap}_n$ by $1$. We proceed as follows: we attach independently to each hyperedge $a$ of $H_{n,\tilde p,h}$ a uniform $[0,1]$ random variable $U_a$. To obtain $H_{n,p,h}$ we remove all hyperedges $a$ such that $U_a\leq \tilde p-p$. This can be done sequentlially by removing at each step the hyperedge corresponding to the lowest remaining $U_a$. Then, at each step, assuming there are still uncovered vertices $v$ in $\tilde A_n$ we can consider the union $K$ of the subgraphs $K(v)$, which has size at least $\epsilon'\tau n$. Hence, with positive probability the hyperedge removed will decrease the value of $\operatorname{gap}_n$. By Chernoff's bound, the number of hyperedges removed is at least $\tau n\frac{\tilde p-p}{2}$ with high probability as $n\to\infty$, therefore $\operatorname{gap}_n$ will reach $0$ with high probability as $n\to\infty$ before we remove all the hyperedges that should be removed. Hence, $H_{n,p,h}$ (and thus $H_{n,\lfloor\tau m\rfloor,h}$) is $(k,l,r)$-orientable a.a.s.
\end{proof}

\end{document}